\pdfoutput=1

\documentclass[letterpaper]{article}

\usepackage{amsmath,amsfonts,amsthm,amssymb, xypic,hyperref, graphicx, cleveref, titlesec, xcolor}

\DeclareMathOperator{\height}{height}

\DeclareMathOperator{\Hom}{Hom}

\newtheorem{theorem}{Theorem}
\newtheorem{prop}[theorem]{Proposition}
\newtheorem{lemma}[theorem]{Lemma}
\newtheorem{corollary}[theorem]{Corollary}
\newtheorem{conj}[theorem]{Conjecture}
\newtheorem{example}[theorem]{Example}
\newtheorem{remark}[theorem]{Remark}
\newtheorem{definition}[theorem]{Definition}
\numberwithin{theorem}{section}

\begin{document}

\title{\bf Convex Polytopes for the Central Degeneration of the Affine Grassmannian}

\author{Qiao Zhou}

\date{}

\maketitle

%post publication version

%fixed one small error in the proof of the main theorem for Iwahori MV cycles. The idea for the proof is certainly correct, but I mixed up $C_1$ vs $C_2$ at some places.

\begin{abstract}

We study the algebraic geometry and combinatorics of the central degeneration (the degeneration that shows up in  local models of Shimura varieties and Gaitsgory's central sheaves) in type A. More specifically, we elucidate the central degeneration of semi-infinite orbits and explain its relations with Levi restriction. Also, we discuss the central degeneration of Mirkovi$\acute{\text{c}}$-Vilonen cycles in the affine Grassmannian, and the corresponding transformations of Mirkovi$\acute{\text{c}}$-Vilonen polytopes. In addition, we shed some light on the geometry of Iwahori MV cycles in the affine Grassmannian and generalized MV cycles in the affine flag variety, which are closely related to Demazure modules and affine Deligne-Lusztig varieties respectively.

\end{abstract}

\tableofcontents

\section{Introduction}

In mathematics, there is an interesting theme of explaining abstract concepts in explicit, concrete and visual ways. This paper is an example of that. In particular, we use toric geometry and some linear algebraic models to study an important construction in geometric representation theory.  

The Geometric Satake Theorem \cite{Lu2, Gin, BD, MV} is a corner stone of the famous Geometric Langlands program. It states that the category of finite dimensional highest weight representations of a reductive group $G^{\vee}$ is equivalent to the category of $G^{\vee}(\mathcal{O})-$equivariant perverse sheaves on the affine Grassmannian for its Langlands dual group $G$. In \cite{MV}, Mirkovi$\acute{\text{c}}$ and Vilonen introduced a collection of algebraic cycles in the affine Grassmannian for $G$, called MV cycles, which corresponds to a basis for the finite-dimensional representations of $G^{\vee}$. Subsequently Anderson initiated the combinatorial study of MV cycles via their torus equivariant moment polytopes, called MV polytopes. Then Kamnitzer \cite{Kam1} gave a complete combinatorial characterization of MV polytopes, and further elucidated the connections between MV polytopes and the canonical bases of representations. 

On the other hand, in \cite{Gai}, Gaitsgory gave a geometric construction of a map from the spherical Hecke algebra to the center of the Iwahori Hecke algebra via a nearby cycles functor from the category of $G(\mathcal{O})-$equivariant perverse sheaves on the affine Grassmannian to the category of Iwahori equivariant perverse sheaves on the affine flag variety. His construction is closely related to local models of Shimura varieties in number theory, as explained in \cite{Ha, HaNg}. The image of this nearby cycles functor is usually called central sheaves. We define the central degeneration to be the $T-$equivariant flat degeneration from the affine Grassmannian to the affine flag variety in the global affine flag variety constructed by Gaitsgory. 

In this paper we study the explicit algebraic geometry and combinatorics of this degeneration, and connect that back to the Geometric Satake correspondence. More specifically, we elucidate the central degeneration of semi-infinite orbits and explain its relations with Levi restriction. Also, we discuss the central degeneration of Mirkovi$\acute{\text{c}}$-Vilonen cycles in the affine Grassmannian, and the corresponding transformations of Mirkovi$\acute{\text{c}}$-Vilonen polytopes. In addition, we shed some light on the geometry of Iwahori MV cycles in the affine Grassmannian and generalized MV cycles in the affine flag variety, which are closely related to Demazure modules and affine Deligne-Lusztig varieties respectively.

\subsection{Main results and future projects}

Consider the global affine flag variety $Fl_{\mathbb{A}^1}$ \cite{Gai} over $\mathbb{A}^1$. Each general fiber over 
$\mathbb{A}^1 \backslash \{ 0 \}$ is isomorphic to the direct product $Gr \times G/B$; its special fiber over $\{ 0 \}$ is isomorphic to the affine flag variety $Fl$.  Let $S$ be a $T-$invariant subscheme of $Gr \times \{ e\}$ in the general fiber, like a $G(\mathcal{O})$ orbit, an MV cycle, an Iwahori orbit, a semi-infinite orbit, an orbit of $T(\mathcal{O})$, etc. We would like to understand the special fiber limit $\tilde{S}$ of $S$ in the affine flag variety.

Recall that in the case of the central degeneration of the closures of $G(\mathcal{O})$ orbits, the special fiber limits are finite unions of Iwahori orbits \cite{Zhu}. In the case of the closures of semi-infinite orbits, we have the following theorem: 

\begin{theorem}

Let $G = GL_n(\mathbb{C})$ or any matrix Lie group over $\mathbb{C}$. Given the closure of any $N_w(\mathcal{K})$ orbit $\overline{S^{\mu}_w}, \mu \in X_*(T)$ in the affine Grassmannian, its special fiber limit is the closure of the corresponding $N_w(\mathcal{K})$ orbit $\overline{S_w^{(\mu, e)}}, (\mu, e) \in W_{aff}$ in the affine flag variety.

\end{theorem}

More generally, let $r_P$ denote the restriction map for a Levi factor $G_J$ of a parabolic subgroup $P^{J}$. Then we show that central degeneration commutes with Levi restriction \cite{BD}. 

\begin{theorem}

Let $S_w^{\mu}, w \in W$ be a semi-infinite orbit in the affine Grassmannian of type A. Let $P^J \supseteq B_w$ be a parabolic subgroup with Levi factor $G_J$. For any $u$ in the Weyl group $W_J$ of $G_J$, the following digram commutes. In other words, central degeneration of the closures of semi-infinite orbits commutes with Levi restriction/parabolic retraction. 

\xymatrixrowsep{10mm} 
\xymatrixcolsep{2pc}
\xymatrix{
\overline{S_{wu}^{\mu}} \subset Gr_G \ar[rr]^{deg} \ar[d]^{r_P^{\mu, u}} & & \overline{S_{wu}^{(\mu, e)}} \subset Fl_G \ar[d]^{r_P^{(\mu, e), u}} \\
 \overline{S_{u, J}^{\mu}} \subset Gr_{G_J} \ar[rr]^{deg} & & \overline{S_{u, J}^{(\mu, e)}} \subset Fl_{G_J}
}

\end{theorem}

Next we discuss the central degeneration of MV cycles \cite{MV} and the related transformations of moment polytopes by presenting two theorems, one for $SL_2(\mathbb{C})$, and the other for general $SL_n(\mathbb{C})$. 

\begin{theorem}

Let $G = SL_2(\mathbb{C})$. Let $S = \overline{Gr^{\lambda} \cap S_{w_0}^{\mu}} = \overline{S_{e}^{- \lambda} \cap S_{w_0}^{\mu}}$, where $\lambda \in X_*(T)^+, \mu \in X_*(T)$, be an MV cycle, and $\tilde{S}$ be its special fiber limit in the affine flag variety. If $\mu = -\lambda$, then $\tilde{S}$ is the $T-$fixed point indexed by $(-\lambda, e)$. Otherwise $\tilde{S}$, has exactly two irreducible components. Each irreducible component of $\tilde{S}$ is a generalized MV cycle and the closure of a GGMS stratum in the affine flag variety. 

More specifically, consider the case of $\mu \neq -\lambda$.  If $\lambda = \mu$, then $S$ is the closure of the $G(\mathcal{O})$ orbit $Gr^{\lambda}$. The two irreducible components of $\tilde{S}$ are the closures of two Iwahori orbits 
$\overline{I^{(\lambda, e)}}= \overline{I^{(\lambda, e)} \cap S_{w_0}^{(\lambda, e)}}$, and $\overline{I^{(-\lambda, e)}} = \overline{I^{(-\lambda, e)} \cap S_{w_0}^{(\lambda, w_0)}}$, both of which are special cases of generalized MV cycles. 
If $-\lambda < \mu < \lambda$, one is the closure of the intersection $I^{(-\lambda, e)} \cap S_{w_0}^{(\mu, w_0)}$,  and the other is the closure of the intersection $I^{(-\lambda + \alpha, w_0)} \cap S_{w_0}^{(\mu, e)}$. 

In both cases, the irreducible component containing $(\mu, e)$ is equal to the closure of the GGMS stratum $S_e^{(-\lambda + \alpha, w_0)} \cap S_{w_0}^{(\mu, e)}$, and the moment polytope of this component is the line interval with vertices $(-\lambda + \alpha, w_0)$ and $(\mu,e)$. The irreducible component containing $(\mu, w_0)$ is equal to the closure of the GGMS stratum $S_e^{(-\lambda, e)} \cap S_{w_0}^{(\mu, w_0)}$, and the moment polytope of this component is the line interval with vertices $(-\lambda, e)$ and $(\mu, w_0)$.

\end{theorem}

\begin{theorem}

Let $G = SL_n(\mathbb{C})$. Given an MV cycle $S = \overline{\cap_{w \in W}S_w^{\mu_w}}$ with special fiber limit $\tilde{S}$. Let $P$ denote the moment polytope of $S$ and $\tilde{S}$ with vertices $(\mu_w, e), w \in W$. Given any irreducible component $S'$ of $\tilde{S}$, its moment polytope $P'$ is a Pseudo-Weyl polytope in $P$. The following two properties must be satisfied by $P'$: (1) every pair of adjacent vertices of $P'$ can be connected by an extended torus invariant $\mathbb{P}^1$; (2) the root number of any vertex $v$ of $P'$, $n_{P'_v}$, is bigger than or equal to the dimension of $S$.  

In $\tilde{S}$, each extremal $T-$fixed point indexed by $(\mu_w, e)$ lies in a unique irreducible component $S'_w$. The moment polytope $P'_w$ of $S'_w$ satisfies the following additional property: (3) the set of adjacent vertices of $(\mu_w, e)$ in $P'_w$, denoted by $C_1$, is in bijection with the set of adjacent vertices of $(\mu_w, e)$ in $P$, denoted by $C_2$. Given a vertex $(\eta, e) \in W_{aff}$ in $C_2$, the corresponding element $f_{\eta}$ in $C_1$ is the affine Weyl group element adjacent to $(\eta, e)$ on the edge connecting $(\mu_w, e)$ and $(\eta, e)$ in $P$.

\end{theorem}

For readers who like examples and diagrams, we explicitly describe all the irreducible components of the special fiber limits of some MV cycles for $G = SL_3(\mathbb{C})$ by illustrating their moment polytopes contained in the original MV polytopes. For example, in the diagram below, the lattice formed by the hyperplanes is the coweight lattice of $G$, and the alcoves are labeled by affine Weyl group elements. Various convex polytopes lie in another lattice, and their vertices all lie in the interior of alcoves. The big trapezoid $P$ is the MV polytope of a given MV cycle $S$, and the five single-colored convex polytopes contained in $P$ are the moment polytopes of the five irreducible components in the special fiber limit of $S$.  Note that in many familiar examples, the moment polytope for each irreducible component contains a unique vertex of the original MV polytope. In this special example below, the black convex polytope does not contain any vertex of $P$, and was discovered by some methods from symplectic geometry. 

\includegraphics[scale = 0.6]{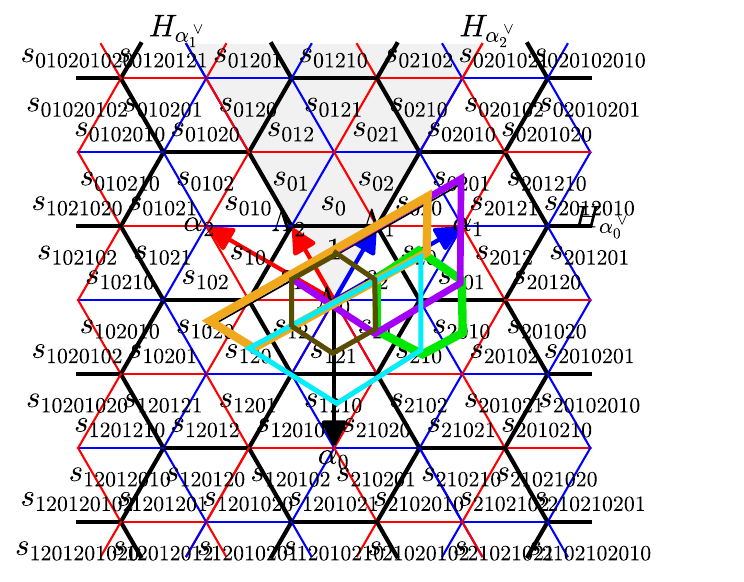}

Throughout this paper, we illustrate the close relations between the special fiber limits of MV cycles and GGMS strata, Iwahori MV cycles (which reflect the combinatorics of Demazure modules for the Langlands dual group) as well as generalized MV cycles (whose geometries are linked to that of affine Deligne-Lusztig varieties). We also prove some results about Iwahori MV cycles and generalized MV cycles by adopting the arguments used in \cite{MV}. 

\begin{theorem}

Let $G$ be a connected reductive algebraic group. Let $\lambda$ and $\mu$ be coweights of $G$, and let 
$\lambda_{dom}$ be the dominant coweight associated with $\lambda$. Let $\tilde{W} = W/W_J$ denote the quotient of the finite Weyl group associated with the partial flag variety $G/P_{\lambda_{dom}}$. Let $\lambda = w \cdot \lambda_{dom}$ for a unique $w \in \tilde{W}$. Let $X_w$ denote the Schubert variety for $w \in \tilde{W}$.

The intersection of the $U^-$ orbit $S_{w_0}^{\mu}$ with the Iwahori orbit $I^{\lambda}$ is equidimensional and of dimension

$$\height(\lambda_{dom} + \mu) - \dim(G/P_{\lambda_{dom}}) + \dim (X_w)$$ 
when $\lambda \leq \mu \leq \lambda_{dom}$, and is $\emptyset$ otherwise. 

\end{theorem}

By the equidimensionality claim in our theorem above and Theorem 11.7 in \cite{Sch}, we conclude the following:

\begin{corollary}
	
The number of irreducible components in the intersection of an Iwahori orbit and a $U^-$ orbit in the affine Grassmannian $I^{\lambda} \cap S_{w_0}^{\mu}, \lambda, \mu \in X_*(T),$ is equal to the $\mu-$weight multiplicity in the Demazure module $B^{\lambda}$ for the Langlands dual group $G^{\vee}$. 

\end{corollary}

One result about generalized MV cycles also immediately follows. 

\begin{corollary}

Let $\gamma' = (w \cdot \lambda_{dom}, w'), \gamma'' = (\mu, w'') \in W_{aff}$, where $\lambda_{dom} \in X_*^+(T), \mu \in X_*(T)$, $\tilde{W}$ is the quotient of $W$ associated with the partial flag variety $G/P_{\lambda_{dom}}$, and $w', w'' \in W, w \in \tilde{W}$. 

The dimension of $I^{\gamma'} \cap S_{w_0}^{\gamma''}$ in the affine flag variety for $G$ is less than or equal to 

$\height(\lambda_{dom} + \mu) - \dim(G/P_{\lambda_{dom}}) + \dim (X_w) + $

$\left\{ \begin{array}{rcl}
 \dim(G/B) - l(w') & \mbox{for} & w' = w'' \in W \\ 
 l(w') - l(w'') & \mbox{for} & w' > w'' \in W 
\end{array} \right.$

If $w'' > w'$, then the intersection is empty. 

\end{corollary}

In the future, it would be very useful to generalize results in this paper to algebraic groups of other types. At the moment, very few examples in other types for the central degeneration are known.

A natural new direction would be to generalize the construction of crystal structures on the set of MV cycles \cite{BrGai, Kam2} to other classes of objects in the affine Grassmannian and affine flag variety, like Iwahori MV cycles and generalized MV cycles. 

We would also like to generalize the work of Kamnitzer \cite{Kam1} and develop a theory of moment polytopes for Iwahori MV cycles and generalized MV cycles in the affine flag variety. That could potentially shed some light on questions related to emptiness and dimensions of affine Deligne-Lusztig varieties. 

In \cite{Gai}, Gaitsgory constructed some central sheaves in $\mathcal{P}_I(Fl)$ by considering the nearby cycles of $\mathcal{P}_{G(\mathcal{O})}(Gr)$. In this paper we discovered that the central degeneration behaves well with respect to semi-infinite orbits. Therefore we would like to apply the nearby cycles functor for the global affine flag variety to the category of $U$ or $U^-$ equivariant perverse sheaves on the affine Grassmannian. 

\subsection{Organization}

The layout of this paper is as follows. Sections 2-3 consist of background material. Section 4 is about the degenerations of semi-infinite orbits. Section 5 contains a discussion of the special fiber limits of MV cycles, by comparing them with GGMS strata and generalized MV cycles in the affine flag variety, and by looking at the moment polytopes of different irreducible components. There is also a main example worked out in details, with lots of illustrative diagrams. In section 6, we include a study of Iwahori MV cycles in the affine Grassmannian and generalized MV cycles in the affine flag variety, as they are closely related to the special fiber limits of MV cycles.

\subsection{Acknowledgments}

I would first like to express my deep gratitude towards my PhD advisor David Nadler for his kind guidance, encouragement and support.  

I am very grateful to Zhiwei Yun for many helpful discussions and meetings for this project. His expertise in the relevant subjects and sharp attention to details greatly propelled me to move forward in my research. 

Moreover, I would like to thank Joel Kamnitzer for explaining to me many important properties of MV cycles and MV polytopes, and for his hospitality during my Toronto visits. I am also indebted to Allen Knutson for teaching me lots of related algebraic geometry and toric geometry. I thank Xinwen Zhu for patiently explaining his work to me, and for suggesting related future projects. I also thank Kevin Costello and Ben Webster for their support and guidance during my time at Perimeter Institute. 

In addition, I benefited from conversations and correspondences with Denis Auroux, Tom Braden, Alexander Braverman, Catherine Cannizzo, Justin Chen, Ben Elias, Davide Gaiotto, Dennis Gaitsgory, Benjamin Gammage, Ulrich G$\ddot{\text{o}}$rtz, Sam Gunningham, Thomas Haines, Xuhua He, Lisa Jeffrey, Thomas Lam, Ian Le, Penghui Li, Dinakar Muthiah, Ivan Mirkovi$\acute{\text{c}}$, James Parkinson, Alexander Postnikov, Arun Ram, Nicolai Reshetikhin, Steven Sam, Vivek Shende, Bernd Sturmfels, Monica Vazirani, Kari Vilonen, Alex Weekes, Lauren Williams and Alex Zorn. 

I would like to sincerely thank the anonymous referee(s) for giving valuable feedback on two earlier versions of this paper. I am also indebted to the support from University of California at Berkeley, Mathematical Sciences Research Institute and Perimeter Institute for Theoretical Physics. 

Research at Perimeter Institute is supported in part by the Government of Canada through the Department of Innovation, Science and Economic Development Canada and by the Province of Ontario through the Ministry of Economic Development, Job Creation and Trade. 

Finally I am very thankful to my family and friends. Their psychological support is invaluable during the ups and downs of this project.

\section{Affine Grassmannian, affine flag variety, and the Central Degeneration}

\subsection{Affine Grassmannian and affine flag variety}

\subsubsection{Loop group and loop algebra}

% data related to $G$ 

Let $G$ be a connected, reductive algebraic group over the field $k = \mathbb{C}$. Let $T$ be a maximal torus of $G$, and let $X^*(T) = \Hom(T, \mathbb{C}^*), X_*(T) = \Hom(\mathbb{C}^*, T)$ and $\Lambda$ denote the weight, coweight and root lattice of $G$ respectively. Let $W = N(T)/T$ denote the Weyl group. Let $B$ be a Borel subgroup of $G$ containing $T$. Let $\alpha_1, ..., \alpha_r$ and 
$\alpha_1^{\vee}, ..., \alpha_r^{\vee}$ denote the simple roots and coroots of $G$ with respect to 
$B$. Let $N$ denote the unipotent radical of $B$. Let $\triangle$ and $\triangle^+$ denote the set of roots and positive roots of $G$.
%For each $i \in V$, let $\psi_i: SL_2 \hookrightarrow G$ denote the $i$th root subgroup of $G$. 

% data related to $LG$

Let $\mathcal{K}$ denote the local field of Laurent series $\mathbb{C}((t))$, and $\mathcal{O}$ denote the ring of formal power series $\mathbb{C}[[t]]$. The group scheme $G(\mathcal{K})$ is also called the loop group $LG$, as we may think of it as a completion of the group of polynomial maps from $S^1 \hookrightarrow \mathbb{C}^*$ to $G$. The group scheme $G(\mathcal{O}) \subset G(\mathcal{K})$ is a maximal compact subgroup of $G(\mathcal{K})$, and is called the positive loop group $L^+G$, as it is the subgroup of the maps $\mathbb{C}^* \rightarrow G$ in $LG$ that can be extended to $0 \in \mathbb{C}$. There is a map $ev_0: G(\mathcal{O}) \rightarrow G$ by evaluating at $t = 0$. Let the Iwahori subgroup $I$ denote the pre-image of $B$ under $ev_0$. When $G = GL_n$, $G(\mathcal{K})$ is the group of invertible matrices with entries in $\mathcal{K}$; $G(\mathcal{O})$ is the group of invertible matrices with entries in $\mathcal{O}$; the Iwahori subgroup $I$ is the subgroup of $G(\mathcal{O})$ whose entries below the diagonal lie in $t\mathcal{O}$. 

% data related to the loop algebra

The loop group has an extended torus $T \times \mathbb{C}^*$ such that $\mathbb{C}^*$ acts as loop rotation by scaling $t$. Its Weyl group is the affine Weyl group $W_{aff}$, which is a semi-direct product of $X_*(T)$ and the finite Weyl group $W$. 
Just like in the case of Lie groups, there is also a notion of Lie algebra and roots for the loop group. Consider the characters of the extended torus. Let $\delta$ denote the character which is trivial on $T$ and identity on the rotation torus. We also view the characters of $T$ as characters on the extended torus, which is trivial on the rotation torus. The roots of $LG$ are given by $ \{ m\delta + \alpha = \mathfrak{g}_{\alpha}z^m, k \in \mathbb{Z}, \alpha \in \Delta \cup \{ 0 \} \}$, which we call affine roots. 
The Lie algebra of the loop group $LG$ is the loop algebra 
$L\mathfrak{g} = (\oplus_{m \in \mathbb{Z}} \mathfrak{t} \cdot z^m) \oplus (\oplus_{(m, \alpha)} \mathfrak{g}_{\alpha}z^m)$.

For further details, see \cite{PS}.

\subsubsection{Affine Grassmannian and affine flag variety}

The affine Grassmannian is the ind-scheme $G(\mathcal{K})/G(\mathcal{O})$. There is a sequence of finite type projective schemes $Gr^i$, $i \in \mathbb{N}$ and closed immersions $Gr^i \hookrightarrow Gr^{i + 1}$, such that for a scheme $S$, $Gr(S) = \lim_{i \rightarrow \infty} Hom(S, Gr^i)$. The affine flag variety is the ind-scheme $G(\mathcal{K})/I$.  

There is an interpretation of the affine Grassmannian and affine flag variety in terms of principal $G-$bundles, due to Beauville-Laszlo. 
The affine Grassmannian $Gr$ is the functor that associates each scheme $S$ the set of pairs $\{ \mathcal{P}, \phi \}$, where $\mathcal{P}$ is an $S-$family of $G-$bundles over the formal disc $\mathbb{D}, \phi: \mathcal{P}|_{\mathbb{D}^*} \rightarrow \mathcal{P}^0|_{\mathbb{D}^*}$ is a trivialization of $\mathcal{P}$ on $\mathbb{D^*}$. The (complete) affine flag variety $Fl$ assigns to each scheme $S$ all the data in $Gr$ plus a $B-$reduction of the principal $G-$bundle $\mathcal{P}$ at $\{0\} \in \mathbb{D}$. There is a natural projection map $Fl \rightarrow Gr$ with fibers being isomorphic to the ordinary flag variety $G/B$.

In type A, there is a combinatorial model for the affine Grassmannian and affine flag variety. Let $V = \mathcal{K}^n$, a vector space over $\mathcal{K}$ with the natural action of $G(\mathcal{K})$. A lattice $L$ in $V$ is a rank $n$ $\mathcal{O}-$submodule such that there exists $N \gg 0, t^NL_0 \subset L \subset t^{-N}L_0$, where $L_0$ denotes the standard $\mathcal{O}$-module $\mathcal{O}^n$. More generally, for any $k-$algebra $R$, an $R-$family of lattices in $V$ is a finitely generated projective 
$R[[t]]-$submodule $M$ of $R((t))^n$ such that $M \otimes_{R[[t]]}R((t)) = R((t))^n$ \cite{Go1, Zhu2}. The relative dimension of a lattice $L$ is given by the difference $\dim(L \backslash L_0) - \dim(L_0 \backslash L)$. 

When $G = GL_n(\mathbb{C})$, the affine Grassmannian for $G$ is isomorphic to the moduli functor which assigns every $k-$algebra $R$ the set of $R-$families of lattices in $k((t))^n$. When $G = SL_n(\mathbb{C})$, then the affine Grassmannian is isomorphic to the moduli spaces of lattices in $V$ with zero relative dimension. Note that the affine Grassmannian for a semisimple group like $SL_n(\mathbb{C})$ is reduced. For further details see \cite{Go1, Mag, PS, Zhu2}. 

Each lattice $L$ can be written as a direct sum 
$L = \mathcal{O}v_1 \oplus \cdots \mathcal{O}v_n$, where $\{ v_1, ..., v_n\}$ is a $\mathcal{K}-$basis of $\mathcal{K}^n$. By choosing the standard basis $\{ e_1, e_2, \cdots, e_n\}$ of $\mathbb{C}^n$, we can represent the lattices in some pictures. Note that $\mathcal{B} = \{ t^m \cdot e_i | m \in \mathbb{Z}, i = 1, 2, ..., n \}$ form a $\mathbb{C}-$basis of $\mathcal{K}^n$, in the sense appropriate to a topological vector space with the $t-$adic topology. For example, when 
$n = 2$, if $L = \mathcal{O} \cdot t^{-1}e_1 \oplus \mathcal{O} \cdot t^2e_2$, $L$ can be represented as below:

\includegraphics[scale = 0.9]{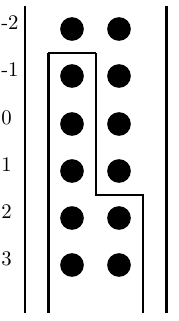}

In the diagram above, each dot represents an element in $\mathcal{B}$. Given a lattice $L$ with a basis $\mathcal{B}_L$, we express each element of $\mathcal{B}_L$ as a linear combination of the elements in $\mathcal{B}$, and the vertical straight lines are used to illustrate the fact that $L$ is an $\mathcal{O}-$module. 

The projective schemes $Gr^i$ consist of all the lattices $L$ such that $t^iL_0 \subset L \subset t^{-i}L_0$. Then $Gr^i$ is a union of ordinary Grassmannians in $(t^{-i}L_0)/(t^iL_0)$ with the extra condition $tS \subset S$ for any subspace $S$ in these ordinary Grassmannians. 

The (complete) affine flag variety $Fl$ is the quotient $G(\mathcal{K})/I$. In type A, there is also a lattice picture for $Fl$. It is the space of complete flags of lattices $L_{\cdot} = (L_1 \supset \cdots \supset L_n)$. 
Each $L_i$ is a lattice such that $L_n \supset tL_1$ and $\dim(L_j / L_{j + 1}) = 1$.

For the rest of this paper, it suffices to consider the affine Grassmannian and affine flag variety as reduced ind-schemes. 

We would like to introduce some interesting orbits in the affine Grassmannian and the affine flag variety by starting from the orbits of affine root subgroups. Given an affine root $\alpha$, we have a one parameter subgroup $U_{\alpha}$ in the loop group $LG$ generated by the exponential of $\alpha$ in the loop algebra $L\mathfrak{g}$.

\begin{prop}\cite{GKM}
\label{rootfixedptsprop}

In the affine Grassmannian or the affine flag variety, every $T \times \mathbb{C}^*-$invariant $\mathbb{P}^1$ is the closure of one orbit of $U_{\alpha}$, for some affine root $\alpha$. In particular, there is a discrete number of one-dimensional $T \times \mathbb{C}^*$ orbits in $Gr$ and $Fl$.

Let $\gamma$ be a $T-$fixed point in the affine Grassmannian or the affine flag variety, $\alpha = \alpha_0 + k \delta$ be an affine root, where $\alpha_0$ is a root of $G$, $k \in \mathbb{Z}$, and $\delta$ is the imaginary root. Let $s_{\alpha}$ be the simple reflection in the affine Weyl group for the hyperplane $H_{\alpha}$, where  
$H_{\alpha} = \{ \beta | <\alpha_0, \beta > = k \}$. Then

 $$\lim_{\eta \rightarrow \infty} \exp(\eta \cdot \alpha) \cdot \gamma = s_{\alpha} \cdot \gamma. $$
 
In other words, there is a unique one-dimensional $U_{\alpha}$ orbit whose closure is a $\mathbb{P}^1$ connecting $\gamma$ and the $T$-fixed point indexed by $s_{\alpha} \cdot \gamma$.
 
\end{prop}

\subsubsection{Semi-infinite orbits}

Semi-infinite orbits in the affine Grassmannian or the affine flag variety are the orbits of the group $U_w = N_w(\mathcal{K})$. Their study is motivated by principal series representations of p-adic groups. They are infinite-dimensional and are indexed by $X_*(T)$ in the affine Grassmannian and $W_{aff}$ in the affine flag variety.  
We view them as ind-varieties. In particular, their intersections with the closure of any $G(\mathcal{O})$ orbit in $Gr$ is an algebraic variety \cite{MV}. For each semi-infinite orbit 
$S_w^{\gamma}, w \in W$ in $Gr$, where $\gamma \in X_*(T)$, or $Fl$, where $\gamma \in W_{aff}$, it is the attracting set of the $T-$fixed point $p^{\gamma}$ indexed by $\gamma$ with the $\mathbb{C}^* \hookrightarrow T$ action given by
$2 w\cdot \hat{\rho}$, where $\hat{\rho}$ is the sum of positive coroots. This gives a decomposition of the affine Grassmannian and affine flag variety into the orbits of $N_w(\mathcal{K})$. 

The Gelfand-Goresky-Macpherson-Serganova (GGMS) strata \cite{Kam1} on the affine Grassmannian or the affine flag variety are the intersections of $|W|$ many semi-infinite orbits. More precisely, given any collection $\{ \alpha_w, w \in W \}$ of elements in $X_*(T)$ or $W_{aff}$, we can form the GGMS stratum $\cap_{w \in W} S_w^{\alpha_w}$.
Note that the closure of each GGMS stratum is an algebraic variety as we view semi-infinite orbits as ind-varieties. 

The following proposition connects GGMS strata with any irreducible $T-$invariant scheme. It essentially follows from section 3.3 in \cite{Kam1}.

\begin{prop}
\label{GGMSirreprop}
Given any $T-$invariant closed irreducible scheme $S$ in the affine Grassmannian or the affine flag variety, there exists a GGMS stratum $\cap_{w \in W} S_w^{\alpha_w}$ whose closure contains $S$, and the $T-$fixed points indexed by $\alpha_w, w \in W$ lie in $S$.
\end{prop}

\begin{proof}
For any $w \in W$, the $\mathbb{C}^*$ action given by the coweight $2 w\cdot \hat{\rho}$ for each $w \in W$ gives a decomposition of $S$ according to the attracting sets of different fixed points. By irreducibility, there exists a unique $T-$fixed point $p^{\alpha_w}$ (indexed by $\alpha_w$) whose attracting set $C_w \subset S$ is top-dimensional. Given any orbit $l$ of $\mathbb{C}^*$ in $S$ with $2 w\cdot \hat{\rho}$ action, $p^{\alpha_w}$ is contained in the closure of $l$. Then $p^{\alpha_w} \in S$ since $S$ is closed. 

For each $w \in W$, the corresponding $C_w$ is dense in $S$, and so is their intersection $\cap_{w \in W} C_w$. 
Therefore, $S = \overline{\cap_{w \in W} C_w} \subseteq \overline{\cap_{w \in W} S_w^{\alpha_w}}$. 
\end{proof}

\subsubsection{Iwahori orbits}

In the affine Grassmannian, Iwahori orbits are indexed by elements in $X_*(T)$, and are vector bundles over ordinary Schubert cells. In the affine flag variety, Iwahori orbits are indexed by $W_{aff}$. The closure of each Iwahori orbit $I^{\gamma}, \gamma \in W_{aff}$ in the affine flag variety is the union of Iwahori orbits $I^{\gamma'}$, where there exists a reduced expression of $\gamma' \in W_{aff}$ which is a sub-word for a reduced expression of $\gamma$. This is the Bruhat order for the affine Weyl group $W_{aff}$, and is in fact independent of the choice of the reduced expressions of $\gamma$.

Because Iwahori orbits are also the orbits of $ev_0^{-1}(N)$, they are affine spaces. Therefore the affine Grassmannian and affine flag variety have a cell decomposition as a union of Iwahori orbits. The homology/cohomology basis given by Iwahori orbits are called the Schubert basis, and the closures of Iwahori orbits are called affine Schubert varieties.

\subsubsection{MV cycles, Iwahori MV cycles and generalized MV cycles}

In \cite{MV}, each MV cycle is defined to be the closure of an irreducible component of the intersection of a $G(\mathcal{O})$ orbit $Gr^{\lambda}$ and a $U^-$ orbit $S_{w_0}^{\mu}$ in the affine Grassmannian, where $\lambda$ is a dominant coweight and $\mu$ is a coweight that belongs to the convex hull of the set $\{ w \cdot \lambda| w \in W \}$.  In \cite{And}, MV cycles for $\overline{Gr^{\lambda}}$, relative to $N$, are the irreducible components of 
$\overline{Gr^{\lambda} \cap S_{e}^{\mu}}$. Equivalently they are the irreducible components of $\overline{S_{e}^{\mu} \cap S_{w_0}^{\lambda}}$. They give a canonical basis of the highest weight representations of the Langlands dual group \cite{MV}. Each MV cycle is also equal to the closure of a GGMS stratum with some extra conditions on the coweights involved \cite{And, Kam1, MV}.  

Similarly, we define an Iwahori MV cycle in the affine Grassmannian to be the closure of an irreducible component of the intersection of an Iwahori orbit $I^{\lambda}, \lambda \in X_*(T)$ and a $U^-$ orbit $S_{w_0}^{\mu}, \mu \in X_*(T)$.

In the affine flag variety, a Generalized MV cycle is defined to be the closure of an irreducible component for the intersection of an Iwahori orbit and a $U^-$ orbit \cite{PRS}. 

\subsection{Central Degeneration}

% moduli description 

Now we would like to define the global analogs of the affine Grassmannian and the affine flag variety, as explained in \cite{Gai}. Let $C$ be a curve with a distinguished point $x$. The global affine Grassmannian over $C$ is the following functor: for any scheme $S$, $\Hom(S, Gr_C)$ is the set of triples $(y, \mathcal{P}, \phi)$, where $y$ is an $S$ point of $C$, $\mathcal{P}$ is a principal $G-$bundles on $C \times S$, and $\phi$ is a trivialization of $\mathcal{P}$ on $C \times S \backslash \Gamma_y$, where $\Gamma_y$ is the graph of $y: S \rightarrow C$. Similarly, the global affine flag variety $Fl_C$ over a curve represents the following functor: for any scheme $S$, $\Hom(S, Fl_C)$ is the set of quadruples $(y, \mathcal{P}, \phi, \zeta)$, where $(y, \mathcal{P}, \phi) \in Gr_C$ and $\zeta$ is a reduction of $\mathcal{P}|_{x \times S}$ to $B \subset G$. We have canonical isomorphisms 
$Fl_{\{ y \} \in C \backslash \{ x \}} \cong Gr_{ \{ y \} } \times G/B$ and $Fl_{ \{ x \}} \cong Fl$.

% group theoretic description
Specializing to the curve $\mathbb{A}^1$, there is a group theoretic definition of the global affine flag variety $Fl_{\mathbb{A}^1}$. It is the moduli space of the pairs $\{ (\epsilon \in \mathbb{A}^1, p \in G(k[t, t^{-1}])/I_{\epsilon}) \}$, where $I_{\epsilon}$ is the pre-image of the Borel subgroup $B$ under the map $G(k[t]) \rightarrow G$ by evaluating at $t = \epsilon \in \mathbb{A}^1$. Topologically $I_{\epsilon}$ are algebraic maps $\mathbb{A}^1 \rightarrow G$ which sends $\{ \epsilon \}$ to $B$. Each fiber $Fl_{\mathbb{A}^1}|_{\epsilon} = G(k[t, t^{-1}])/I_{\epsilon}$ has a natural map to the affine Grassmannian $Gr = G(k[t, t^{-1}])/G(k[t])$. When $\epsilon \neq 0$, $Fl_{\mathbb{A}^1}|_{\epsilon}$ has a map to $G/B$ by evaluating at $t = \epsilon$, so it is isomorphic to $Gr \times G/B$. When $\epsilon = 0$, the fiber $G(k[t, t^{-1}])/I_{0}$ is isomorphic to the affine flag variety. Equivalently, $Fl_{\mathbb{A}^1} = \{ (\epsilon \in \mathbb{A}^1, p \in G(k[(t - \epsilon), (t - \epsilon)^{-1}])/I) \}$.

We have the following commutative diagram:

\xymatrix{
Gr \times G/B \ar[d] \ar@{^{(}->}[r] &Gr \times G/B \times (\mathbb{A}^1 \backslash \{ 0 \})\ar[d] \ar@{^{(}->}[r] & Fl_{\mathbb{A}^1} \ar[d] & Fl \ar[d] \ar@{_{(}->}[l]\\
\{ \epsilon \neq 0  \} \ar@{^{(}->}[r] & \mathbb{A}^1 \backslash \{ 0\} \ar@{^{(}->}[r] & \mathbb{A}^1 & \{0 \} \ar@{_{(}->}[l]} 

%lattice model

In type A, there is also a lattice picture of $Fl_{\mathbb{A}^1}$. When $G = GL_n(\mathbb{C})$, the global affine flag variety over $\mathbb{A}^1$(with its reduced structure) is isomorphic to the moduli space of triples 
$(\epsilon, L, f)$, where $\epsilon \in \mathbb{A}^1$, $L$ is a rank $n$ $k[t]$ module in $k[t, t^{-1}]^n$ , and $f$ is a flag in the quotient $L/(t - \epsilon)L$. When $\epsilon = 0$, we recover the lattice picture of the affine flag variety in type A. 
When $\epsilon \neq 0$, there is a canonical isomorphism $L/(t - \epsilon)L \cong k^n$, and we recover a lattice in the affine Grassmannian and a flag in $k^n$. Equivalently, the lattice model of $Fl_{\mathbb{A}^1}$ in type A (with its reduced structure) is isomorphic to the space of triples $(\epsilon, L, f)$, where $\epsilon \in \mathbb{A}^1$, $L$ is a rank $n$ $k[(t - \epsilon)]$-module in $k[(t - \epsilon), (t - \epsilon)^{-1}]^n$, and $f$ is a flag in the quotient $L/tL$.

% define central degeneration

We aim to understand this construction more explicitly by examining how some interesting torus invariant spaces in the affine Grassmannian degenerate.

\begin{definition}

The central family $F_{\text{cen}}$ in the global affine flag variety is the flat family with general fibers being $Gr \times \{ e \}$ in $Fl_{\mathbb{A}^1}$. The central degeneration of a $T-$invariant scheme $S$ in the affine Grassmannian is the degeneration of $S \times \{ e \} \subset Gr \times \{ e \} \subset Gr \times G/B \cong Fl_{\{ \epsilon \neq 0 \}}$ in the global affine flag variety $Fl_{\mathbb{A}^1}$. The special fiber limit $\tilde{S}$ of $S$ under the central degeneration is defined to be the intersection of the closure of a family of $S$ in the general fibers of $Fl_{\mathbb{A}^1}$ with the special fiber $Fl_{\{ 0 \}}$, i.e. $\overline{S \times (\mathbb{A}^1 \backslash \{ 0\})} \cap Fl_{\{ 0 \}}$.

\end{definition}

\begin{remark}

When the underlying curve $C$ for the global affine flag variety is not $\mathbb{A}^1$, we can define the central degeneration of a $T-$invariant scheme $S$ when $S$ is invariant under $Aut(\mathcal{D})$. 

\end{remark}

% motivation for central degeneration

The term ``central degeneration" comes from Gaitsgory's construction of central sheaves. Algebraically, the center of the Iwahori affine Hecke algebra $H_I = F_c(I \backslash G(\mathcal{K}) / I)$ is isomorphic to the spherical Hecke algebra 
$H_{sph} = F_c( G(\mathcal{O}) \backslash G(\mathcal{K}) / G(\mathcal{O}))$, with the algebra structure given by convolution of compactly-supported functions. 
In \cite{Gai} Gaitsgory constructed a map from $H_{sph}$ to $Z(H_I)$ geometrically, through the nearby cycles functor from the category of $G(\mathcal{O})$ equivariant perverse sheaves on the affine Grassmannian $P_{G_{\mathcal{O}}}(Gr)$ to the category of Iwahori equivariant perverse sheaves on the affine flag variety $\mathcal{P}_I(Fl)$.  

Gaitsgory's construction involves the degeneration of $Gr \times \{ e \} \subset Gr \times G/B$ in the global affine flag variety. This degeneration corresponds to a linear algebraic model \cite{Go, Ha, HaNg} from local models of Shimura varieties in number theory, as explained below.

\begin{definition}[G. Laumon]

Let $n_{-} \leq 0 \leq n_+$ be two integers, and $d = n_+ - n_-$. 

Let $M_{r, n_{\pm}}$ be the functor which associates to each $\mathcal{O}-$algebra $R$ the set of $n-$tuples $L_{\cdot} = (L_0, ..., L_{n - 1})$ where $L_0, ..., L_{n - 1}$ are $R[t]$ submodules of $t^{n_-}R[t]^n/t^{n_+}R[t]^n$, satisfying the following properties:

\begin{itemize}
  \item as $R-$modules, $L_0, ..., L_{n - 1}$ are locally direct factors with rank $nd - r$ in $t^{n_-}R[t]^n/t^{n_+}R[t]^n$. The positive integer $r$ is the dimension of each $L_i, i = 0, ..., n - 1$, as an $R-$submodule,
  
  \item $\gamma(L_0) \subseteq L_{1}, \gamma(L_1) \subseteq L_{2}, ..., \gamma(L_{n-1}) \subseteq L_{0}$, where $\gamma$ is the matrix 
 $$
\begin{bmatrix}
0 &  1  & \ldots & 0 \\
\vdots  & \ddots & \ddots & \vdots \\
\vdots & \ddots & \ddots & 1\\
t + \epsilon  &   0      &\ldots & 0
\end{bmatrix}.
$$
\end{itemize}

\end{definition}

%\begin{remark}
%
%The linear algebraic model defined above is equivalent to the model with the same data except that the second condition is changed to $L_i \supset \gamma_i \cdot L_{i - 1} (i~mod~n)$, where $\gamma_i$ is the $n \times n$ diagonal matrix with 
%$i$th entry being $t - \epsilon$ and all the other entries being $1$.
%
%\end{remark}

The linear algebraic model defined above embeds in the global affine flag variety.

%The isomorphism below only holds when we consider all possible pairs of integers $n_{\pm}$ as well as $r$
\begin{prop}
The linear algebraic model defined above is isomorphic to the subfamily in the lattice model for the global affine flag variety by fixing the flag $f$ to be the standard flag when $\epsilon \neq 0$.
\end{prop}

\begin{proof}

For each $\mathbb{C}[t]$ module $L_i$ in the collection $L_{\cdot}$ specified in the functor $M_{r, n_{\pm}}(\mathbb{C})$, it can be identified with a lattice $L'_i$ in $\mathbb{C}[t, t^{-1}]^n$ such that 
$t^{d_1}\mathbb{C}[t]^n \subset L'_i \subset t^{d_2}\mathbb{C}[t]$ and $d_2 - d_1 = d$.
From the lattices $L'_{\cdot} = \{ L'_0,...,L'_{n-1} \}$ we can get another sequence of lattices
$L'_0 \supset \gamma_n \cdot L'_{n - 1} \supset \gamma_n \cdot \gamma_{n-1} \cdot L'_{n - 2} \supset \cdots \supset (t - \epsilon) L'_0$. When $\epsilon \neq 0$, this sequence specifies a lattice in $Gr$ and the standard flag. When $\epsilon = 0$, this sequence specifies a point in the affine flag variety. We can define a map from the lattice model to $M_{r, n_{\pm}}(\mathbb{C})$ in an analogous way. The same argument works when we replace $\mathbb{C}$ with any $\mathbb{C}$-algebra $R$. Therefore, they are isomorphic. 

\end{proof}

%Later on we are going to show that all $\mathbb{P}^1$s invariant under the extended torus would degenerate in this way (when the underlying curve is $\mathbb{A}^1$). 

Now let's make a few more remarks about the central degeneration of the closures of $G(\mathcal{O})$ orbits, which was studied for Gaitsgory's central sheaves \cite{Gai}. 

Given $\gamma_1, \gamma_2 \in W_{aff}$, $\gamma_2$ is $\gamma_1-$admissible if there exists $w \in W$ such that $\gamma_2 \leq w \cdot \gamma_1$ according to the usual Bruhat order. In \cite{Zhu}, it was proved that the special fiber limit of the closure of a $G(\mathcal{O})$ orbit $\overline{Gr^{\lambda}}$ in the affine Grassmannian is the union of Iwahori orbits $I^{\lambda'}$ in the affine flag variety, where $\lambda'$ is a $(\lambda, e)-$admissible element in the affine Weyl group. 

When $G = GL_n$ and $\lambda$ is a minuscule coweight, $Gr^{\lambda}$ is an ordinary Grassmannian. In \cite{Go}, the explicit equations for the central degeneration of these ordinary Grassmannians are obtained from studying the linear algebraic functor $M_{r, n_{\pm}}$ described above by focusing on the $d = 1$ case. 

An interesting basic example was worked out for $G = GL_2(\mathbb{C})$ or $PGL_2(\mathbb{C})$ in \cite{Gai}.

\begin{example}

Consider the closure of the minuscule $G(\mathcal{O})-$orbit $Y_0 = \{$lattices $L$ which are contained in the standard lattice $\mathcal{O} \oplus \mathcal{O}$ with $\dim(\mathcal{O} \oplus \mathcal{O} \backslash L) = 1 \}$. By construction, $Y_0$ is isomorphic to $\mathbb{P}_1$. The special fiber limit of $Y_0 \times \{e\}$ in the affine flag variety $Fl$ is isomorphic to a union of two $\mathbb{P}^1$s intersecting at a point. Locally, this is the family $\{ xy = c | c \in k \}$. 

\end{example}

\subsection{Global group scheme}

We discussed earlier that the affine Grassmannian and affine flag variety admit a natural action of the loop group $G(\mathcal{K})$ and its subgroups. As for the global affine flag variety $Fl_C$, we can similarly define global group schemes which act naturally on it. Recall the moduli description of the global affine flag variety $Fl_C$ \cite{Gai}. There exists a global group scheme $G(\mathcal{K})_{C}$ that acts on $Fl_C$. When restricted to each fiber of $Fl_{C}$ above a point $c \in C$, $G(\mathcal{K})_{C}$ acts like $G(\mathcal{K})$ on the trivializations of a principal $G-$bundle $\mathcal{P}_G$ over the formal punctured disk near 
$c \in \mathcal{D}^*_c$. It also acts on the $B-$reduction at the special point by changing the gluing data of the principal $G-$bundle $\mathcal{P}_G$. For some subgroups of $G(\mathcal{\mathcal{K}})$ such global versions also exist. 

%when invoking explicit global coordinates, use polynomials in terms of (t - \epsilon) and (t - \epsilon)^{-1}...DON'T USE LAURENT SERIES IN GLOBAL COORDINATES...that is purely local
More concretely, if we fix the curve to be $\mathbb{A}^1$ and $G = GL_n(\mathbb{C})$ or any matrix Lie group, then $G(\mathcal{K})_{\mathbb{A}^1}$ can be represented by an $\mathbb{A}^1$ family of matrices with entries in $\mathbb{C}[(t - \epsilon), (t - \epsilon)^{-1}], \epsilon \in \mathbb{A}^1$. In this subsection, we will discuss global versions of $T \subset G$, $G(\mathcal{K})$, $G(\mathcal{O})$, $N_w(\mathcal{K})$ and affine root subgroups in this special case. 

When $G = GL_n(\mathbb{C})$, the global group scheme $G(\mathcal{K})_{\mathbb{A}^1}$ is the $\mathbb{A}^1-$family of nonsingular matrices $G([(t - \epsilon), (t - \epsilon)^{-1}]), \epsilon \in \mathbb{A}^1$. When $G = SL_n(\mathbb{C})$, $G(\mathcal{K})_{\mathbb{A}^1}$ is the subgroup in $GL_n(\mathcal{K})_{\mathbb{A}^1}$ with determinant-one matrices. We can extend this construction of $G(\mathcal{K})_{\mathbb{A}^1}$ to any matrix Lie group. Similarly, the global group scheme $G(\mathcal{O})_{\mathbb{A}^1}$ is the subgroup in $G(\mathcal{K})_{\mathbb{A}^1}$ such that the matrices have $\mathbb{C}[(t - \epsilon)], \epsilon \in \mathbb{A}^1$ coefficients. 

%Describe the action in terms of the group theoretic model of $Fl_{\mathbb{A}^1}$

Each matrix in $G(\mathcal{K})_{\mathbb{A}^1}$ acts naturally on $Fl_{\mathbb{A}^1}$. More explicitly, given $\epsilon' \in \mathbb{A}^1$, $G([(t - \epsilon'), (t - \epsilon')^{-1}])$ acts naturally on the fibers 
$G([(t - \epsilon), (t - \epsilon)^{-1}])/I, \epsilon \in \mathbb{A}^1$ of the global affine flag variety. When $\epsilon = \epsilon'$, this action near $\epsilon$ is equal to the natural action of $G(\mathcal{K})$ on the affine flag variety if $\epsilon = 0$, and is equal to the natural action of $G(\mathcal{K}) \times G(\mathcal{K})|_{t = 0}$ on $Gr \times G/B$ if $\epsilon \neq 0$. When $\epsilon \neq \epsilon'$, the action of $G([(t - \epsilon'), (t - \epsilon')^{-1}])$ near $\epsilon$ can be seen more clearly by rewriting its matrix coefficients as Taylor expansions near $\epsilon$. 

\begin{prop}
\label{globalTcoro}

The global group scheme $G(\mathcal{K})_{\mathbb{A}^1}$ has a global subgroup $G_{\mathbb{A}^1}$. It acts as $G$ on each fiber of the global affine flag variety, and therefore its action is constant. 

Similarly, the global version of the maximal torus $T \subset G$, $T_{\mathbb{A}^1}$, also has a constant action on the global affine flag variety. 

\end{prop}

Next we would like to construct global versions of $N_w(\mathcal{K})$ and global affine root subgroups. 

For each $w \in W$, the group $N_w(\mathcal{K})$ is a finite product of groups of the form $U_{\alpha_0}(\mathcal{K})$, where $\alpha_0$ is a finite root of $G$.  For each $U_{\alpha_0}(\mathcal{K})$, there is a filtration by subgroups 
$U_{\alpha_0, m}(\mathcal{K}) = \{ x \in U_{\alpha_0}(\mathcal{K})| val(x) \geq m, m \in \mathbb{Z}\}$. 
Also note that $U_{\alpha_0}(\mathcal{K}) = \lim_{m \rightarrow - \infty} U_{\alpha_0}(\mathcal{K})$. All of these remain true when we replace Laurent series coefficients with $\mathbb{C}[(t - \epsilon), (t - \epsilon)^{-1}], \epsilon \in \mathbb{A}^1$ coefficients. 

%All of these remain true for $U_{\alpha_0}(\mathbb{C}[(t - \epsilon), (t - \epsilon)^{-1}]), \forall \epsilon \in \mathbb{A}^1$.  

When $G = GL_n(\mathbb{C})$ or any matrix Lie group, the global group schemes $N_w(\mathcal{K})_{\mathbb{A}^1}, w \in W$ is the global subgroup of $G(\mathcal{K})_{\mathbb{A}^1}$ whose matrix elements belong to 
$N_w([(t - \epsilon), (t - \epsilon)^{-1}]), \epsilon \in \mathbb{A}^1$. Similarly, for any finite root $\alpha_0$ and integer $m$, there exist global subgroups $U_{\alpha_0}(\mathcal{K})_{\mathbb{A}^1}$ and  $U_{\alpha_0, m}(\mathcal{K})_{\mathbb{A}^1}$ whose matrix entries belong to $U_{\alpha_0}([(t - \epsilon), (t - \epsilon)^{-1}]), \epsilon \in \mathbb{A}^1$ and $U_{\alpha_0, m}([(t - \epsilon), (t - \epsilon)^{-1}]), \epsilon \in \mathbb{A}^1$ respectively. 
The relations between $N_w(\mathcal{K})$, $U_{\alpha_0}(\mathcal{K})$ and $U_{\alpha_0, m}(\mathcal{K})$ still hold for their global versions.

In \cite{Zhu}, there are related discussions of global group schemes in a more general setting. 

\section{Moment polytopes for the Central Degeneration}

In this section, we lay the foundation for building connections between the central degeneration and toric geometry. We first introduce alcoves following \cite{PS}, as well as the alcove lattice as defined in \cite{LP1}. Our convex polytopes lie in the alcove lattice. Then we discuss torus equivariant moment polytopes for the global affine flag variety. 

\subsection{Alcove lattice}

In ordinary Lie theory one usually thinks of the weight lattice $X^*(T)$, and hence the roots, as lying in the real vector space $\mathfrak{t}^*_{\mathbb{R}}$. For affine roots, one can think of them as linear forms on the Lie algebra 
$\mathbb{R} \times \mathfrak{t}_{\mathbb{R}}$ of the extended torus. However, it is more convenient to regard them as affine-linear functions on $\mathfrak{t}_{\mathbb{R}}$. Then for nonzero root $\alpha$, the affine root $m \delta + \alpha$ is determined up to sign by the affine hyperplane $$H_{m, \alpha} = \{  \xi \in \mathfrak{t}_{\mathbb{R}}: (\alpha, \xi) = -m\}$$ in $\mathfrak{t}_{\mathbb{R}}$ where it vanishes. This collection of hyperplanes is called the diagram of $LG$, and lie on the coweight lattice. 

The connected components in $\mathfrak{t}_{\mathbb{R}}$ of the complement of the hyperplanes $H_{m, \alpha}$ are called the alcoves of the diagram. Recall that in ordinary Lie theory, the connected components in the complement of the $H_{0, \alpha}$ (which form the diagram of $G$) are called the chambers. Each chamber $C$  contains a unique alcove $C_0$ whose closure contains the origin. If one chooses a chamber $C$ then the roots of $G$ are called positive or negative according to their sign on $C$. Similarly, an affine root is called positive or negative according to its sign on $C_0$. The positive affine roots corresponding to the walls of $C_0$ are called the simple affine roots \cite{PS}. 

The affine Weyl group $W_{aff} = N(T \times \mathbb{C}^*)/(T \times \mathbb{C}^*)$ is isomorphic to a semi-direct product of the coweight lattice and the finite Weyl group. It acts transitively on the set of alcoves when $G$ is simply connected, e.g. $G = SL_n(\mathbb{C})$ or $SU_n(\mathbb{C})$. The $A_1$ and $A_2$ root systems with alcoves labeled by elements in the corresponding affine Weyl groups are shown below.  

\includegraphics[scale = 0.5]{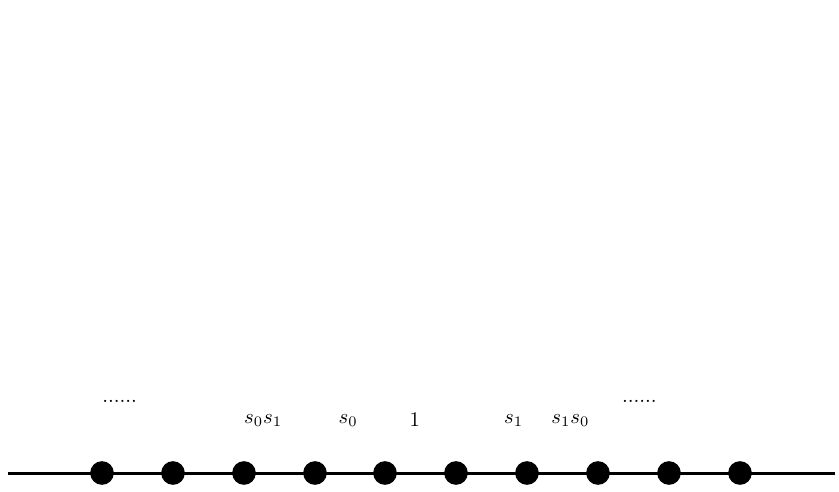}

\includegraphics[scale = 0.5]{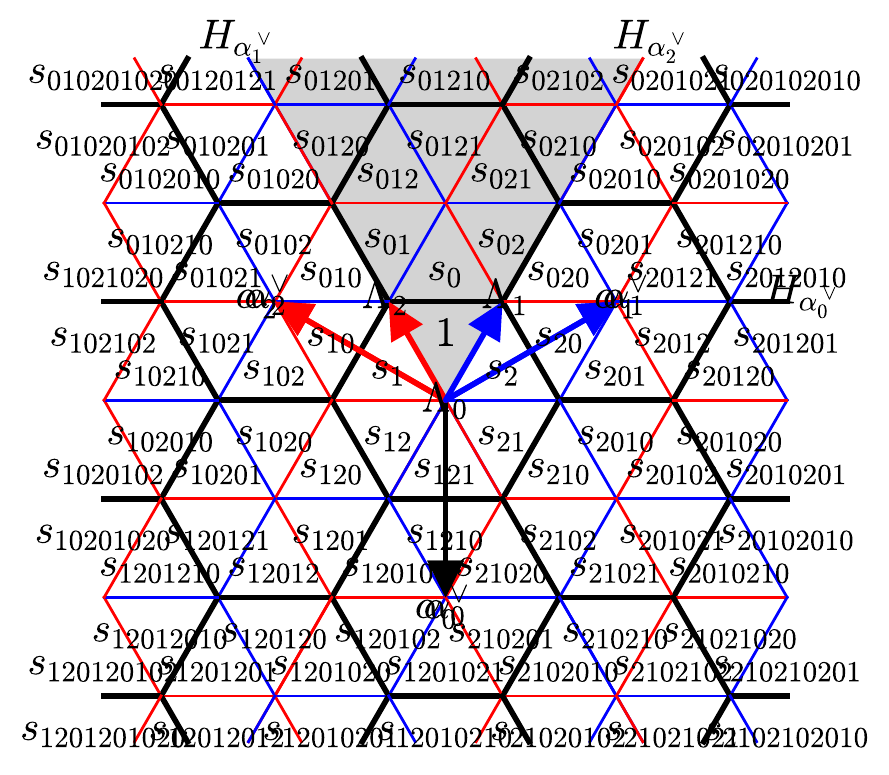}

As defined in section 3.2 of \cite{LP1}, the alcove lattice is the infinite graph whose vertices correspond to alcoves, and edges correspond to pairs of alcoves that are separated by a wall. It embeds in the weight lattice of $G$. 

\subsection{Moment polytopes}

%moment map for the global affine flag variety

The maximal torus $T$ in $G$ has a constant action on each fiber of the global affine flag variety $Fl_{\mathbb{A}^1}$. On the other hand, the rotation torus $\mathbb{C}^*$ scales the base curve $\mathbb{A}^1$. Therefore its action does not preserve individual fibers. We would like to define the $T-$equivariant moment map on the global affine flag variety. 

Let $\mathcal{L}$ be an ample line bundle on the global affine flag variety $Fl_C$, where $C$ is a curve. Let $\Gamma(Fl_C, \mathcal{L}^*)$ be the vector space of global sections of $\mathcal{L}^*$. Then $Fl_C$ embeds in the projective space $\mathbb{P}(V)$, where $V = \Gamma(Fl_C, \mathcal{L})^*$, by mapping $x \in Fl_C$ to the point determined by the line in $V$ dual to the hyperplane $\{ s \in \Gamma(Fl_C, \mathcal{L})| s(x) = 0\}.$ 

Let $T_{\mathbb{R}}$ denote the real compact form of $T$. The moment map $\Phi_{T_{\mathbb{R}}}$, for the action of the torus $T_{\mathbb{R}} \subset T \subset G$, is a map from $Fl_C \hookrightarrow \mathbb{P}(V)$ to $\mathfrak{t}^*_{\mathbb{R}}$. When we restrict $\Phi_{T_{\mathbb{R}}}$ to individual fibers in $Fl_C$, we get a map from $Gr \times G/B$ or $Fl$ to $\mathfrak{t}^*_{\mathbb{R}}$. There is an embedding of the weight lattice $X^*(T)$ to $\mathfrak{t}^*_{\mathbb{R}}$. Following the arguments in section 7 of \cite{Yun}, with appropriate choices, the moment map image of each $T-$fixed point in the fibers $Gr \times G/B$ or $Fl$ is a lattice point of the embedded weight lattice. The moment polytope of a $T-$invariant projective scheme is the convex hull of the images of its $T-$fixed points \cite{Bri2, GM}.  

The moment polytopes of $T-$invariant schemes in the global affine flag variety can be identified with certain convex polytopes in the alcove lattice. In particular, the moment map image of a $T-$fixed point indexed by $\gamma = (\lambda, w) \in W_{aff}$ in the affine flag variety lies in the interior of the alcove that corresponds to $\gamma \in W_{aff}$. Similarly, the moment map image of a $T-$fixed point indexed by $\gamma = (\lambda, w) \in X_*(T) \times W$ in the general fiber $Gr \times G/B$ lies in the interior of the alcove that corresponds to $(\lambda, w) \in X_*(T) \times W$. 

On the affine Grassmannian and the affine flag variety, there is also a $T-$equivariant moment map $\Phi_{T_{\mathbb{R}}}$ which is equal to the restriction of the one defined on the global affine flag variety to individual fibers of the central family. The $T-$equivariant moment polytope of the closure of a GGMS stratum $\overline{\cap_{w \in W} S_w^{\alpha_w}}$ in the affine Grassmanian or the affine flag variety is the convex hull of $\{ \alpha_w, w \in W\}$ in the alcove lattice. This is an example of a Pseudo-Weyl polytope \cite{Kam1}. In particular, the $T-$equivariant moment polytope of an MV cycle (the closure of a special GGMS stratum) is called an MV polytope \cite{And}. There is also a $T \times \mathbb{C}^*$ equivariant moment map $\Phi_{(T \times \mathbb{C}^*)_{\mathbb{R}}}$ on the affine Grassmannian and the affine flag variety. 

\begin{definition}
Given a $T-$invariant scheme $S$ with $T-$equivariant moment polytope $P$, a $T-$fixed point $p$ of $S$ is an extremal $T-$fixed point if its moment map image is a vertex of $P$. 

For each vertex $v$ of a convex polytope $P$ in the alcove lattice, we define the root number of $v$ in $P$, $n_{P_v}$, to be the number of lattice points $p'$ in $P$ which satisfy the following conditions: (1) $p'$ lies on the same lattice line as $v$; (2) $p'$ and $v$ can be connected by an extended torus invariant $\mathbb{P}^1$, or equivalently the closure of an affine root subgroup orbit. If $S$ is invariant under the extended torus $T \times \mathbb{C}^*$, then this is the number of $T \times \mathbb{C}^*$ invariant curves in $S$ which pass through the $T-$fixed point whose moment map image is $v$. 
\end{definition}

Below we explain some preliminary relations between the dimension of a $T \times \mathbb{C}^*$ invariant scheme in the affine Grassmannian or the affine flag variety and their moment polytopes. 

First, we recall an important result by Brion in \cite{Bri}, section 1.4. 

\begin{prop}[Brion]
\label{localdimprop}

Let $H$ be a torus acting on a variety $X$ with an isolated fixed point $x$, such that the number of closed irreducible $H-$stable curves through $x$ is finite; denote this number by $n(X, x)$. 
Let $\dim_x(X)$ denote the dimension of the tangent space $T_xX$ at $x$. Then $\dim_x(X) \leq n(X, x)$. 

\end{prop}

The corollary below applies to many interesting varieties in the affine Grassmannian and affine flag variety, like affine Schubert varieties, MV cycles, Iwahori MV cycles, generalized MV cycles, etc. 

\begin{corollary}

Let $X$ be an irreducible algebraic variety in the affine Grassmannian or the affine flag variety that is invariant under the action of the extended torus. Let $P$ be its $T \times \mathbb{C}^*$-equivariant moment polytope, $\Gamma$ be its moment graph for the extended torus, and $P'$ be its $T-$equivariant moment polytope. Then the dimension of $X$ is less than or equal to the number of edges incident at any vertex of $\Gamma$. Equivalently, the dimension of $X$ is less than or equal to the root number of any vertex of $P'$. 

\end{corollary}

\subsection{Torus equivariant flat families} % see Dec 2017 Math Diaries. Hartshorne book has relevant information. 

By \cite{Har}, given a morphism $f: X \rightarrow Y$, where $X$ is an integral scheme, $Y$ is a nonsingular curve, and $f$ is surjective, then $f$ is a flat morphism. Since we will focus on families of MV cycles (which are irreducible and reduced \cite{MV}) in the global affine flag variety over $\mathbb{A}^1$, we are dealing with flat families in this paper. 

A family of schemes over $\mathbb{A}^1$ is called $T-$equivariant if $T$ acts on each fiber and its action commutes with the scaling action of $\mathbb{C}^*$ on $\mathbb{A}^1$. Then any flat family of $T-$invariant schemes in the global affine flag variety $Fl_{\mathbb{A}^1}$ is $T-$equivariant, as the action of $T$ is constant on all the fibers by \Cref{globalTcoro}. 

\begin{prop}

Consider a $T-$equivariant flat family of projective schemes, let $S$ be a $T-$invariant projective scheme in the general fiber, and let $\tilde{S}$ denote its limit in the special fiber. Then the $T-$equivariant moment polytope of $\tilde{S}$ coincides with that of $S$. 

\end{prop}

\begin{proof}

For a $T-$equivariant flat family of projective varieties, the multi-graded Hilbert polynomial is constant. Then the Duistermaat-Heckman measure on $\mathfrak{t}^*$, being the leading order behavior of the multi-graded Hilbert polynomial, also stays constant. The moment polytopes, which is the support of the Duistermaat-Heckman measure on $\mathfrak{t}^*$, is constant as well.  

For more details, see \cite{GLS, Har}. 

\end{proof}

\begin{corollary}
\label{Tfixeddegcoro}

Under the central degeneration, the special fiber limit of a $T-$fixed point in $Gr \cong Gr \times \{ e \}$ indexed by $\mu \in X_*(T)$ is the $T-$fixed point in $Fl$ indexed by the translation element $(\mu, e) \in W_{aff}$.
\end{corollary}

\begin{proof}
All distinct $T-$fixed points in every fiber of the global affine flag variety have distinct moment map images. Since this degeneration is $T-$equivariant, the special fiber limit of a $T-$fixed point in the general fiber must be the $T-$fixed in $Fl$ with the same moment map image. For the $T-$action on the global affine flag variety, the moment map images of the $T-$fixed point indexed by $\mu \in X_*(T)$ in $Gr \times \{ id \}$ and the $T-$fixed point indexed by $(\mu, e) \in W_{aff}$ in $Fl$ are the same. Therefore, the latter must be the special fiber limit of the former. 

\end{proof}

\begin{corollary}
\label{flatpolytopelemma}

The $T-$equivariant moment polytopes of a flat family of $T-$invariant schemes in the global affine flag variety stay the same. In other words, the moment polytope for the special fiber limit is the convex hull of the limits of the vertices of the moment polytope for the general fibers. 

\end{corollary}

There are two other related facts that we would like to recall. 

\begin{lemma}

Consider a flat family of schemes over a curve. Let $S_1$ and $S_2$ be two closed subschemes in the general fibers (which are isomorphic). Then the special fiber limit of the intersection of $S_1$ and $S_2$, $S_1 \cap S_2$, is contained in the intersection of the special fiber limit of $S_1$ and the special fiber limit of $S_2$. 

\end{lemma}

\begin{prop}
\label{flatdimlemma}

Consider a flat family of projective schemes in which the general fibers are irreducible of dimension $d$, then the special fiber is (set-theoretically) equidimensional of dimension $d$. 

\end{prop}

\begin{proof}

This follows from Chapter 3, Corollary 9.6 of \cite{Har}.

\end{proof}

\section{Degenerations of semi-infinite orbits and Levi Restriction}

Now let's consider the central degeneration of semi-infinite orbits, which are the orbits of $U_w = N_w(\mathcal{K})$. We start by discussing closure relations of semi-infinite orbits in the affine flag variety. Then we introduce some global group schemes, which act on the global affine flag variety and some flat families that we are interested in. Afterwards we show that the special fiber limit of the closure of a semi-infinite orbit in the affine Grassmannian is the closure of a semi-infinite orbit in the affine flag variety that is indexed by a translation element in the affine Weyl group. This suggests a more elementary proof of the existence of a weight functor for Gaitsgory's central sheaves as discussed in \cite{AB}. Finally, we consider the Levi restriction functor on the global affine flag variety, and show that central degeneration commutes with Levi restriction/parabolic retraction.

\subsection{Closure relations of semi-infinite orbits in the affine flag variety}

We describe the closure relations of semi-infinite orbits in the affine flag variety. In this subsection we will focus on the orbits of $U^-$, but closure relations of the orbits of other $U_w$ can be worked out in completely analogous ways. 

Let $G$ be a complex connected reductive algebraic group. Recall that in the affine Grassmannian for $G$, the $U^-$ orbit 
$S_{w_0}^{\mu'}$ is in the closure of $S_{w_0}^{\mu}$ if and only if $\mu' \leq \mu$ in $X_*(T)$. Similarly in the affine flag variety, closure relations between semi-infinite orbits are given by the semi-infinite Bruhat order, or Lusztig's generic order \cite{Lu}, which we describe below. 

\begin{prop}[Alcove Picture] 

When $G = SL_2(\mathbb{C})$, the semi-infinite Bruhat order is given below:

$$\cdots < s_1s_0s_1 < s_1s_0 < s_1 < e < s_0 < s_0s_1 < s_0s_1s_0 < \cdots.$$

For general $G$, given a pair $w, w' \in W_{aff}$, $w' \leq w$ if and only if $w'$ is contained in the closed cone with vertex $w$ generated by the negative coroots.

\end{prop}

\begin{example}
Let $G = SL_3(\mathbb{C})$. The green cone below is the moment map image of the closure of the $U^-$ orbit $S_{w_0}^{e}$ in the affine flag variety. A $U^-$-orbit $S_{w_0}^{\gamma}$ is in the closure of $S_{w_0}^{e}$ if and only if the lattice point corresponding to $\gamma$ in the alcove lattice lies in this cone. 

\includegraphics[scale = 0.5]{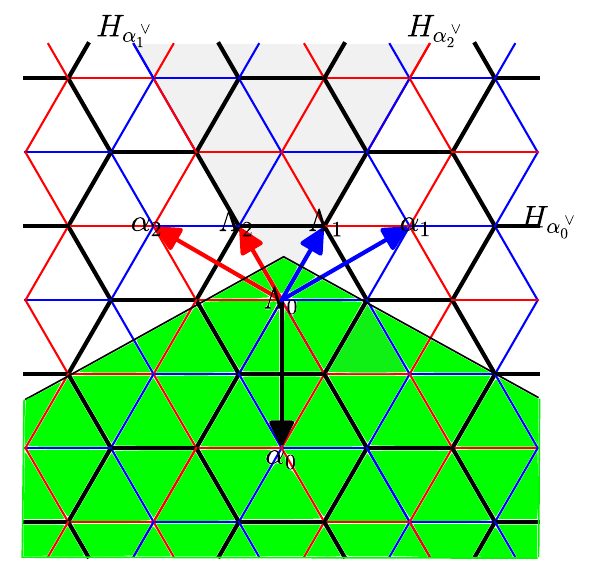}

\end{example}

Below we state an equivalent description of the periodic Bruhat order in the type A lattice picture. 

\begin{prop} [Lattice Picture for type A]

Let $G = GL_n(\mathbb{C})$ or $SL_n(\mathbb{C})$. Let $L_{\cdot} = (L_0 \supset L_1 \supset \cdots \supset L_{n-1} \supset t\cdot L_0)$ and 
$L'_{\cdot}= (L'_0 \supset L'_1 \supset \cdots \supset L'_{n-1} \supset t\cdot L'_0)$ be two sequences of coordinate lattices which correspond to two 
$T-$fixed points in the affine flag variety. Let $\gamma$ and $\gamma'$ denote the affine Weyl group elements indexed by $L_{\cdot}$ and $L'_{\cdot}$ respectively. Similarly, let $\eta_i$ and $\eta'_i$ denote the coweights for $GL_n(\mathbb{C})$ corresponding to the lattices $L_i$ and $L'_i$  $\forall i = 0, ..., n - 1$. Then $\gamma \leq \gamma'$ if and only if $\eta_i \leq \eta'_i$ as coweights $\forall i = 0, ..., n - 1$.

\end{prop} 

In this paper there are two relevant ordering on elements of $W_{aff}$. For the relation ``less than or equal to", we denote $\leq_{B}$ for the usual Bruhat ordering (the closure relations for Iwahori orbits), and $\leq_{U_w}$ for the periodic Bruhat ordering that corresponds to the closure relations for $U_w$ orbits. The same notations apply to the other ordering relations.

\subsection{Global group schemes that act on the Central Family}

Earlier on we introduced various global group schemes that act on the global affine flag variety. In this subsection, we will focus on their global subgroups which act on the central family $F_{\text{cen}}$. These are special cases of the related constructions in \cite{Zhu}. 

\begin{prop}
\label{globalgrouplemma}

Let $G = GL_n(\mathbb{C})$ or any matrix Lie group over $\mathbb{C}$. The global group scheme $G(\mathcal{K})_{\mathbb{A}^1}$ has a global subgroup $G(\mathcal{K})'_{\mathbb{A}^1}$ that acts on the central family $F_{\text{cen}}$. A matrix $M$ in 
$G(\mathcal{K})_{\mathbb{A}^1}$ with $\mathbb{C}[(t - \epsilon), (t - \epsilon)^{-1}], \epsilon \in \mathbb{A}^1$ coefficients is an element of $G(\mathcal{K})'_{\mathbb{A}^1}$ if and only if it satisfies the following condition: every entry in $M$ that is below the diagonal is of the form $t \cdot l, l \in \mathbb{C}[(t - \epsilon), (t - \epsilon)^{-1}]$. 

\end{prop}

\begin{proof}

Since the global subgroup $G(\mathcal{K})'_{\mathbb{A}^1}$ of $G(\mathcal{K})_{\mathbb{A}^1}$ acts on $F_{\text{cen}}$, it is necessary and sufficient that when $\epsilon \neq 0$, the matrices in $G(\mathcal{K})'_{\mathbb{A}^1}$ with 
$\mathbb{C}[(t - \epsilon), (t - \epsilon)^{-1}]$ coefficients preserve the standard $B-$reduction of a principal $G-$bundle at $0$. This means that entries below the diagonal of these matrices are all $0$ when we set $t = 0$. When $\epsilon = 0$, the matrices in $G(\mathcal{K})'_{\mathbb{A}^1}$ with $\mathbb{C}[t, t^{-1}] $ coefficients has to be the limit of the general fibers of $G(\mathcal{K})'_{\mathbb{A}^1}$ by letting $\epsilon \rightarrow 0$, in order to satisfy the gluing conditions of a global scheme. 

\end{proof}

\begin{theorem}
\label{smallpiecetheorem}

Let $G = GL_n(\mathbb{C})$ or any matrix Lie group over $\mathbb{C}$. The global group scheme $U_{\alpha_0, m}(\mathcal{K})_{\mathbb{A}^1}$ has a global subgroup $U_{\alpha_0, m}(\mathcal{K})'_{\mathbb{A}^1}$ which acts on the central family $F_{\text{cen}}$.

Case 1: $\alpha_0$ is a positive root. For any $\epsilon \in \mathbb{A}^1$, the restriction of $U_{\alpha_0, m}(\mathcal{K})'_{\mathbb{A}^1}$ to the fiber above $\epsilon$, $U_{\alpha_0, m}(\mathcal{K})'_{\mathbb{A}^1}|_{\{ \epsilon \}}$, is equal to $U_{\alpha_0, m}(\mathbb{C}[(t - \epsilon), (t - \epsilon)^{-1}])$. Their actions on individual fibers of the central family coincide. 

Case 2: $\alpha_0$ is a negative root. When $\epsilon \neq 0$, $U_{\alpha_0, m}(\mathcal{K})'_{\mathbb{A}^1}|_{\{ \epsilon \}} \subset U_{\alpha_0, m}(\mathbb{C}[(t - \epsilon), (t - \epsilon)^{-1}])$. For any $T-$fixed point $t^{\lambda}, \lambda \in X_*(T)$, in the general fiber of $F_{\text{cen}}$, the two orbit closures below coincide: $$\overline{U_{\alpha_0, m}(\mathcal{K})'_{\mathbb{A}^1}|_{\{ \epsilon \}} \cdot t^{\lambda} }= \overline{U_{\alpha_0, m}(\mathbb{C}[(t - \epsilon), (t - \epsilon)^{-1}]) \cdot t^{\lambda}}.$$ 
When $\epsilon = 0$, $U_{\alpha_0, m}(\mathcal{K})'_{\mathbb{A}^1}|_{\{ \epsilon \}} =  U_{\alpha_0, m+1}(\mathbb{C}[t, t^{-1}])$. 
The actions of these two groups on the special fiber coincide.

\end{theorem}

\begin{proof}

The existence and matrix description of $U_{\alpha_0, m}(\mathcal{K})'_{\mathbb{A}^1}$ follow from the arguments in the proof of \Cref{globalgrouplemma}. 

If $\alpha_0$ is a positive root, the only entry of any matrix element in $U_{\alpha_0, m}(\mathcal{K})'_{\mathbb{A}^1}$ lies above the diagonal. Therefore it is the same as $U_{\alpha_0, m}(\mathcal{K})_{\mathbb{A}^1}$.

If $\alpha_0$ is a negative root, the only entry of any element in $U_{\alpha_0, m}(\mathcal{K})'_{\mathbb{A}^1}$ lies below the diagonal and has the form $t \cdot l$, where $l \in \mathbb{C}[(t - \epsilon), (t - \epsilon)^{-1}]$ and $val(l) \geq m$. 
When $\epsilon = 0$, every entry of the form $t \cdot (a(t - \epsilon)^m + ...)$ becomes $at^{m + 1} + ...$. Therefore, $U_{\alpha_0, m}(\mathcal{K})'_{\mathbb{A}^1}|_{\{ 0 \}}$ coincide with $U_{\alpha_0, m+1}(\mathbb{C}[t, t^{-1}])$. 
When $\epsilon \neq 0$, $U_{\alpha_0, m}(\mathcal{K})'_{\mathbb{A}^1}|_{\{ \epsilon \}} \subset U_{\alpha_0, m}(\mathbb{C}[(t - \epsilon), (t - \epsilon)^{-1}])$ as $t \cdot l = \epsilon \cdot l + (t - \epsilon) \cdot l$. 
Given a $T-$fixed point $t^{\lambda}$ in the general fiber, there exists an integer $v$ such that $U_{\alpha_0, v}(\mathbb{C}[(t - \epsilon), (t - \epsilon)^{-1}])$ belongs to the stabilizer of $t^{\lambda}$ but 
$U_{\alpha_0, v - 1}(\mathbb{C}[(t - \epsilon), (t - \epsilon)^{-1}])$ doesn't. If $m \geq v$, 
then $U_{\alpha_0, m}(\mathcal{K})'_{\mathbb{A}^1}|_{\{ \epsilon \}} \cdot t^{\lambda} = U_{\alpha_0, m}(\mathbb{C}[(t - \epsilon), (t - \epsilon)^{-1}]) \cdot t^{\lambda} = t^{\lambda}$. 
If $m < v$, the orbit of the set of matrix elements with entries of the form $c_m(t - \epsilon)^m + \cdots + c_{v-1}(t - \epsilon)^{v-1}, c_i \neq 0, m \leq i \leq v-1$ is dense in $U_{\alpha_0, m}(\mathbb{C}[(t - \epsilon), (t - \epsilon)^{-1}]) \cdot t^{\lambda}$. Consider any matrix $M_f \in U_{\alpha_0, m}(\mathbb{C}[(t - \epsilon), (t - \epsilon)^{-1}])$ in the aforementioned set with the only entry being $f(t) = a_m(t - \epsilon)^m + \cdots + a_{v-1}(t - \epsilon)^{v-1}$. 
Let $M_g \in U_{\alpha_0, m}(\mathcal{K})'_{\mathbb{A}^1}|_{\{ \epsilon \}}$ be the matrix whose only entry is $g(t) = t(b_m(t - \epsilon)^m + \cdots + b_{v-1}(t - \epsilon)^{v-1}) \in U_{\alpha_0, m}(\mathcal{K})'_{\mathbb{A}^1}|_{\{ \epsilon \}}$ such that $b_m, ..., b_{v-1}$ satisfy the finite collection of linear equations $\{ \epsilon b_m = a_m, b_{m} + \epsilon b_{m+1} = a_{m + 1}, ..., b_{v-2} + \epsilon b_{v-1} = a_{v-1} \}$. Then $g(t) = a_m(t - \epsilon)^m + \cdots + a_{v-1}(t - \epsilon)^{v-1} + b_{v-1}(t - \epsilon)^v$. 
Therefore $M_f \cdot t^{\lambda} = M_g \cdot t^{\lambda}$. Since set of orbits of the elements of the form $M_f$ is dense in $U_{\alpha_0, m}(\mathbb{C}[(t - \epsilon), (t - \epsilon)^{-1}]) \cdot t^{\lambda}$, we conclude that 
$\overline{U_{\alpha_0, m}(\mathcal{K})'_{\mathbb{A}^1}|_{\{ \epsilon \}} \cdot t^{\lambda} }= \overline{U_{\alpha_0, m}(\mathbb{C}[(t - \epsilon), (t - \epsilon)^{-1}]) \cdot t^{\lambda}}$.

\end{proof}

\begin{corollary}
\label{rootdegcoro}

Let $G = GL_n(\mathbb{C})$ or any matrix Lie group over $\mathbb{C}$. The global group scheme $U_{\alpha_0}(\mathcal{K})_{\mathbb{A}^1}$ has a global subgroup $U_{\alpha_0}(\mathcal{K})'_{\mathbb{A}^1}$ which acts on $F_{\text{cen}}$.

If $\alpha_0$ is a positive root, $U_{\alpha_0}(\mathcal{K})'_{\mathbb{A}^1}|_{\{ \epsilon \}} = U_{\alpha_0}(\mathbb{C}[(t - \epsilon), (t - \epsilon)^{-1}])$ for all $\epsilon \in \mathbb{A}^1$, and their actions on individual fibers of $Fl_{\mathbb{A}^1}$ coincide. 

If $\alpha_0$ is a negative root, for any $\epsilon \in \mathbb{A}^1$, $U_{\alpha_0}(\mathcal{K})'_{\mathbb{A}^1}|_{\{ \epsilon \}} \subset U_{\alpha_0}(\mathbb{C}[(t - \epsilon), (t - \epsilon)^{-1}])$. For any $T-$fixed point $p$ in the $\epsilon-$fiber of $F_{\text{cen}}$, the closure of $U_{\alpha_0}(\mathcal{K})'_{\mathbb{A}^1}|_{\{ \epsilon \}} \cdot p$ coincides with the closure of 
$U_{\alpha_0}(\mathbb{C}[(t - \epsilon), (t - \epsilon)^{-1}]) \cdot p$.

\end{corollary}

\begin{proof}

The global group scheme $U_{\alpha_0}(\mathcal{K})'_{\mathbb{A}^1}$ is the limit of $U_{\alpha_0, m}(\mathcal{K})'_{\mathbb{A}^1}$ as $m \rightarrow - \infty$. Then the claim follows from \Cref{smallpiecetheorem}. 

\end{proof}

\begin{corollary}
\label{Nwdegcoro}

Let $G = GL_n(\mathbb{C})$ or any matrix Lie group over $\mathbb{C}$. For any $w \in W$, the global group scheme $N_w(\mathcal{K})_{\mathbb{A}^1}$ has a subgroup $N_w(\mathcal{K})'_{\mathbb{A}^1}$ which acts on $F_{\text{cen}}$, and they agree if and only if $w$ is the identity element in $W$. 

Given $\epsilon \in \mathbb{A}^1, w \in W$ and any $T-$fixed point $p$ in the $\epsilon-$fiber of $F_{\text{cen}}$, the closure of $N_w(\mathcal{K})'_{\mathbb{A}^1}|_{\{ \epsilon \}} \cdot p$ is isomorphic to the closure of $N_w(\mathcal{K}) \cdot p$.

\end{corollary}

\begin{proof}

For each $w \in W$, $N_w(\mathcal{K})_{\mathbb{A}^1}$ is a finite product of $U_{\alpha_0}(\mathcal{K})_{\mathbb{A}^1}$ for some finite roots $\alpha_0$. Then the claim immediately follows from \Cref{rootdegcoro}.

\end{proof}

Finally, we are ready to state a theorem about the degeneration of semi-infinite orbits. 

\begin{theorem}
\label{semiinfinitethrm}

Let $G = GL_n(\mathbb{C})$ or any matrix Lie group over $\mathbb{C}$. Given the closure of any $N_w(\mathcal{K})$ orbit, $\overline{S^{\mu}_w}, \mu \in X_*(T)$, in the affine Grassmannian, its special fiber limit is the closure of the corresponding $N_w(\mathcal{K})$ orbit, $\overline{S_w^{(\mu, e)}}, (\mu, e) \in W_{aff}$, in the affine flag variety.

\end{theorem}

\begin{proof}

By \Cref{Tfixeddegcoro}, the $T-$fixed point $t^{\mu}, \mu \in X_*(T)$ in the general fibers $Gr \times \{ e\}$ degenerate to the $T-$fixed point in the special fiber $Fl$ indexed by the translation element $(\mu, e) \in W_{aff}$. Then the closure of the semi-infinite orbit in the affine Grassmannian $\overline{S_w^{\mu}} = \overline{N_w(\mathcal{K})'_{\mathbb{A}^1}|_{\{\epsilon \neq 0\}} \cdot t^{\mu}}$ degenerates to the closure of the semi-infinite orbit in the affine flag variety 
$\overline{S_w^{(\mu, e)}} = \overline{N_w(\mathcal{K})'_{\mathbb{A}^1}|_{\{0 \}} \cdot t^{(\mu, e)}}$ by \Cref{{Nwdegcoro}}. 
 
\end{proof}

\begin{remark}
\Cref{semiinfinitethrm} above shows that the special fiber limit of a semi-infinite orbit in the affine Grassmannian has to be a semi-infinite orbit in the affine flag variety which corresponds to a translation element in the affine Weyl group. 
\end{remark}

Now we discuss an application of \Cref{semiinfinitethrm} above. 

In \cite{AB} Theorem 4, the authors connect Gaitsgory's central sheaves with geometric Satake by constructing a weight functor on the central sheaves and proving the following commutative diagram.

\begin{theorem}[Arkhipov-Bezrukavnikov]
\label{wtfunctorthrm}

Let $\Phi_1$ denote the nearby cycles functor for Gaitsgory's central sheaves: $\Phi_1: \mathcal{P}_{G(\mathcal{O})}(Gr_G) \rightarrow \mathcal{Z}(\mathcal{P}_I(Fl_G))$. 

Let $\Phi_2$ denote the hyperbolic localization functor $H^*_c(S_e^{\mu}, \mathcal{F}), \mu \in X_*(T)$ from $\mathcal{P}_{G(\mathcal{O})}(Gr_G)$ to $\mathcal{P}_{T(\mathcal{O})}(Gr_T)$, which is equivalent to the tensor category of $T^{\vee}$ representations $Rep(T^{\vee})$. 

Let $\Phi_3$ denote the hyperbolic localization functor on $\mathcal{P}_I(Fl_G)$ given by 
$H^*_c(S_e^{\gamma}, \mathcal{F}), \gamma \in W_{aff}, \mathcal{F} \in \mathcal{P}_I(Fl_G)$. 
% double check which $w \in W$ for $S_w^{\gamma}$

The following diagram commutes:

\xymatrixrowsep{10mm} 
\xymatrixcolsep{2pc}
\xymatrix{
\mathcal{P}_{G(\mathcal{O})}(Gr_G) \ar[rr]^{\Phi_1} \ar[rd]^{\Phi_2} & & \mathcal{P}_I(Fl_G)) \ar[ld]^{\Phi_3} \\
 & \mathcal{P}_{T(\mathcal{O})}(Gr_T) \cong Rep(T^{\vee})
}

\end{theorem}

Roughly speaking, our \Cref{{Nwdegcoro}} and \Cref{semiinfinitethrm} indicates that any nearby fiber of a semi-infinite orbit $S_e^{\gamma} \subset Fl_G$ for a non-translation element $\gamma \in W_{aff}$ admits a free $T-$action, whereas the nearby fiber of a semi-infinite orbit $S_e^{\gamma} \subset Fl_G$ for a translation element $\gamma = (\mu, e)$ contains a unique $T-$fixed point indexed by $\mu \in X_*(T)$. From the definition of cohomology with compact support, this suggests a more elementary proof of \Cref{wtfunctorthrm}.

\subsection{Levi restriction}

Note that the functors $\Phi_2$ and $\Phi_3$ in \Cref{wtfunctorthrm} are hyperbolic localization functors, and they involve an appropriate choice of $\mathbb{C}^* \rightarrow T$ which retracts each semi-infinite orbit in the affine Grassmannian and affine flag variety of $G$ to a unique $T-$fixed point, as we discussed in 2.1.3. In general, given any cocharacter $\mathbb{C}^* \rightarrow T$, there is a Bialynicki-Birula decomposition for the affine Grassmannian and affine flag variety, where the retracting cells are some semi-infinite orbits, and the fixed subscheme is isomorphic to the affine Grassmannian for a Levi subgroup. This is called parabolic retraction \cite{BaGau} or Levi restriction \cite{BD}, and has interesting relations to representation theory. In this subsection, we introduce Levi restriction for the affine Grassmannian, as developed in \cite{BD}, and extend it to the affine flag variety in type A. Then we explain the relationship between Levi restriction and central degeneration. Majority of our notations are adopted from those used in \cite{Kam2}.

Fix $w \in W$. To study the retraction where the attracting cells are $N_w(\mathcal{K})$ orbits for some $\mathbb{C}^* \hookrightarrow T$ action, we consider parabolic subgroups which contain the Borel $B_w = wBw^{-1}$. 

Let $P^J \subseteq B_w$ be a parabolic subgroup. Here $J$ is the subset of the set of vertices of the Dynkin diagram such that the Lie algebra of $P^J$ consists of all the root spaces 
$\mathfrak{g}_{\alpha}$ with $\alpha \in w \cdot \Delta^+ \cup \Delta_J$, where $\Delta_J \subset \Delta$ is the root subsystem generated by the simple roots $\{ \alpha_j: j \in J \}$. Let $N^J$ be its unipotent radical. Let $G_J = P^J/N^J$ be the corresponding Levi factor. Then the quotient map $q: P^J \rightarrow G_J$ has a canonical splitting $s$. This naturally induces a map $q: Gr_{P^J} \rightarrow Gr_{G_J}$ with canonical splitting $s$. 

The image of $B_w$ under the quotient $q$ is denoted by $B_J$ and its unipotent radical is denoted $N_J$. Under the canonical splitting $s$,  the image of $N_J$ is the unipotent subgroup of $G$ (also denoted $N_J$) whose Lie algebra contains all the weight spaces $\mathfrak{g}_{\alpha}$ for $\alpha \in \Delta_J \cap w \cdot \Delta^+$. Let $W_J$ denote the Weyl group of $G_J$. Under $s$, it is identified with the subgroup of $W$ generated by $\{ s_j: j \in J\}$. Let $v \in W$ denote the longest element of $W_J$. For $u \in W_{J}$ we have the unipotent subgroup $N_{u, J} = uN_Ju^{-1}$ of $G_J$. Similarly, its image under the canonical splitting is the unipotent subgroup of $G$ whose Lie algebra contains all the weight spaces 
$\mathfrak{g}_{\alpha}$ for $\alpha \in \Delta_J \cap wu \cdot \Delta^+$. For $u \in W_J$ and $\mu \in X_*(T)$, let $S_{u, J}^{\mu}$ denote the corresponding semi-infinite orbit in $Gr_{G_J}$. 

The inclusion $P^J \hookrightarrow G$ induces a morphism $i: Gr_{P^J} \rightarrow Gr_G$ which is bijective, but its inverse is not continuous in general \cite{BD}. However, it is a homeomorphism when restricted to certain semi-infinite orbits. 
Given $\lambda \in X_*(T), u \in W_J$ and any $g \in N_{wu}(\mathcal{K})$, the map $i^{-1}$ on $S_{wu}^{\lambda}$ is given by 
$i^{-1}(g \cdot t^{\lambda}) = g \cdot t^{\lambda} \in Gr_{P^J}$. To see that this is indeed well-defined, consider $g_1, g_2 \in N_{wu}(\mathcal{K})$ such that $g_1 \cdot t^{\lambda}$ and $g_2 \cdot t^{\lambda}$ are the same point in $Gr_G$. Then there exists $h \in G(\mathcal{O})$ such that $g_1 \cdot t^{\lambda} = g_2 \cdot t^{\lambda} \cdot h$. Since the left hand side of the equation belongs to $P^J(\mathcal{K})$, $h \in P^J(\mathcal{K}) \cap G(\mathcal{O}) = P^J(\mathcal{O})$. Therefore, $g_1 \cdot t^{\lambda}$ and $g_2 \cdot t^{\lambda}$ are the same point in $Gr_{P^J}$. Thus, $i$ is a homeomorphism on the semi-infinite orbits $S_{wu}^{\lambda}, u \in W_J, \lambda \in X_*(T)$ of $Gr_G$ and $Gr_{P^J}$. 

We have the restriction morphism $r_P^{\lambda, u}: S_{wu}^{\lambda} \rightarrow S_{u, J}^{\lambda}$, which is defined to be the composition $i^{-1} \circ q$. It can be identified with a Bialynicki-Birula retraction map on the affine Grassmannian $Gr_G$. We can extend the map $r_P$ to semi-infinite orbits on the affine flag variety, as semi-infinite orbits in the affine flag variety undergo Bialynicki-Birula retractions as well. Here we present a construction in type A. 

Each point of the affine flag variety of $GL_n(\mathbb{C})$ can be represented as a sequence of lattices $L_{\cdot} = (L_0 \supset L_1 \supset \cdots \supset L_{n-1} \supset t\cdot L_0)$. Given a semi-infinite orbit $S_w^{\gamma}, \gamma \in W_{aff}$, it can be represented as $N_w(\mathcal{K}) \cdot L_{\cdot}$, where $L_{\cdot}$ is a sequence of coordinate lattices. Each lattice $L_i$ corresponds to a coweight $\mu_i$ and represents a $T-$fixed point on one connected component of the affine Grassmannian for $GL_n$. 

Given $\gamma \in W_{aff}, u \in W_J$, let $L_{\cdot}$ be the sequence of coordinate lattices representing $\gamma$ in a connected component of $Fl_{GL_n(\mathbb{C})}$. We have the restriction map $r_P^{\gamma, u}$ on $S_{wu}^{\gamma} \subset Fl_G$, which on the level of $\mathbb{C}$ points is given by $r_P^{\gamma, u}(g \cdot L_{\cdot}) = q(g) \cdot L_{\cdot}$ for any $g \in N_{wu}(\mathcal{K})$. Here the group elements acts on $L_{\cdot}$ by acting on individual lattices simultaneously. Let $L'_{\cdot}$ be a sequence of lattices in $S_w^{\gamma}$, then $r_P(L'_{\cdot}) = (r_P(L'_0) \supseteq r_P(L'_1) \supseteq \cdots \supseteq r_P(L'_{n-1}) \supseteq t\cdot r_P(L_0))$. Note that some inclusions in the sequence of projected lattices $r_P(L'_{\cdot})$ must be strict, and this new sequence of lattices is isomorphic to a point in $Fl_{G_J} = G_J(\mathcal{K})/(I \cap G_J(\mathcal{O}))$. 

In the following theorem, we explain Levi restriction for the global affine flag variety. For some related sheaf-theoretic results, see \cite{HaRi}.

\begin{theorem}
\label{restrideg}

Let $S_w^{\mu}, w \in W$ be a semi-infinite orbit in the affine Grassmannian of type A. Let $P^J \supseteq B_w$ be a parabolic subgroup with Levi factor $G_J$. For any $u$ in the Weyl group $W_J$ of $G_J$, the following digram commutes. In other words, central degeneration of the closures of semi-infinite orbits commutes with Levi restriction/parabolic retraction. 

\xymatrixrowsep{10mm} 
\xymatrixcolsep{2pc}
\xymatrix{
\overline{S_{wu}^{\mu}} \subset Gr_G \ar[rr]^{deg} \ar[d]^{r_P^{\mu, u}} & & \overline{S_{wu}^{(\mu, e)}} \subset Fl_G \ar[d]^{r_P^{(\mu, e), u}} \\
 \overline{S_{u, J}^{\mu}} \subset Gr_{G_J} \ar[rr]^{deg} & & \overline{S_{u, J}^{(\mu, e)}} \subset Fl_{G_J}
}

\end{theorem}
 
\begin{proof}
 
The maps $r_P^{\mu, u}$ and $r_P^{(\mu, e), u}$ have a canonical splitting $s' = i^{-1} \circ s$. Then $s'(S_{u, J}^{\mu}) =  N_{u, J} \cdot \mu \subseteq S_{wu}^{\mu}$ and $s'(S_{u, J}^{(\mu}, e)) =  N_{u, J} \cdot (\mu, e) \subseteq S_{wu}^{(\mu, e)}$. By \Cref{rootdegcoro} and the arguments used in the proof of \Cref{semiinfinitethrm}, the special fiber limit of the closure of the orbit $\overline{N_{u, J} \cdot (\mu, e)} \subseteq \overline{S_{wu}^{(\mu, e)}}$ is the closure of the corresponding orbit 
$\overline{N_{u, J} \cdot (\mu, e)} \subseteq \overline{S_{wu}^{(\mu, e)}}$. Therefore, the diagram in the claim commutes. 

\end{proof}

\section{Degenerations of Mirkovi$\acute{\text{c}}$-Vilonen cycles}

In this section, we explain the central degeneration of MV cycles and the corresponding transformations of MV polytopes. 

First of all, our knowledge of the degenerations of the closures of $G(\mathcal{O})$ orbits and semi-infinite orbits immediately gives us some useful ideas about the special fiber limits of MV cycles.

\begin{prop}
\label{basicMVlemma}

Let $S$ be an MV cycle in $Gr^{\lambda} \cap S_{w_0}^{\mu}$ for some $\lambda \in X_*(T)^+, \mu \in X_*(T)$, and $\tilde{S}$ be its special fiber limit. Let $\mu_w, w \in W$ be such that $S = \overline{\cap_{w \in W} S^{\mu_w}_w}$. Then $\tilde{S}$ is contained in the union of non-empty intersections of Iwahori orbits and $U^-$ orbits $I^{w_1} \cap S_{w_0}^{w_2}$ in the affine flag variety, such that $w_1, w_2 \in W_{aff}$ are both $\lambda-$admissible, and $w_2 \leq_{U^-} (\mu, e)$. 

Moreover, $\tilde{S}$ is contained in the intersection of the closures of semi-infinite orbits in the affine flag variety, $\cap_{w \in W} \overline{S_w^{(\mu_w, e)}}$, which is a finite union of GGMS strata. Each irreducible component $S'$ of $\tilde{S}$ is contained in the closure of a GGMS stratum $S_1$ in $\cap_{w \in W} \overline{S_w^{(\mu_w, e)}}$, such that $S'$ and $S_1$ have the same moment polytope. 

\end{prop}

\begin{proof}

These claims immediately follow from \cite{Zhu}, \Cref{semiinfinitethrm} and \Cref{GGMSirreprop}.

\end{proof}

\subsection{$SL_2(\mathbb{C})$ case}

In this subsection, we aim to give an explicit description of the special fiber limits of MV cycles in the affine Grassmannian for $SL_2(\mathbb{C})$, and compare that with generalized MV cycles and GGMS strata in the affine flag variety for 
$SL_2(\mathbb{C})$. First, we recall and prove some basic properties of GGMS strata and generalized MV cycles in the affine flag variety, and then we prove the main theorem for this subsection and illustrate it through an example. 

When $G = SL_2(\mathbb{C})$, MV cycles are a bit simpler than the general case. In this case, each MV cycle is the closure of the intersection $Gr^{\lambda} \cap S_{w_0}^{\mu}$ or the GGMS stratum $\overline{S_{e}^{- \lambda} \cap S_{w_0}^{\mu}}$ for some $\lambda \in X_*(T)^+, \mu \in X_*(T)$. This corresponds to the fact that in each finite dimensional irreducible representation of $PGL_2(\mathbb{C})$, the multiplicity of each weight is one.

We are going to describe the GGMS strata and generalized MV cycles for $G = SL_2(\mathbb{C})$. The proposition below shows that lots of GGMS strata in the affine flag variety are empty. 

\begin{prop}
\label{semiinfiniteemptylemma}
In the affine flag variety for $SL_2(\mathbb{C})$, a GGMS stratum involving two semi-infinite orbits indexed by distinct translation elements $S_e^{(\mu, e)} \cap S_{w_0}^{(\mu', e)}, \mu \neq \mu'$ is empty. The same is true if the two semi-infinite orbits are indexed by distinct non-translation elements $S_e^{(\mu, w_0)} \cap S_{w_0}^{(\mu', w_0)}, \mu \neq \mu'$. 
\end{prop}

\begin{proof}

Suppose there exist distinct translation elements $(\mu, e), (\mu', e) \in W_{aff}$ such that the intersection $S_e^{(\mu, e)} \cap S_{w_0}^{(\mu', e)}$ is nonempty, then it is necessary that $\mu' \geq \mu$. Otherwise the intersections of the moment map images of $S_e^{(\mu, e)}$ and $S_{w_0}^{(\mu', e)}$ would be empty. 

We represent $(\mu, e)$ and $(\mu', e)$ as sequences of coordinate lattices $L_0 \supset L_1$ and $L'_0 \supset L'_1$. Then $L_0$ and $L'_0$  are indexed by $\mu, \mu' \in X_*(T)$. Similarly, $L_1$ and $L'_1$ are indexed by two coweights 
$\gamma, \gamma'$ for $GL_2(\mathbb{C})$. The loop group $G(\mathcal{K})$ and its various subgroups act on the space of sequences of lattices by acting on different lattices simultaneously. 

Now consider this intersection in terms of sequences of lattices. In the closure of the intersection $\overline{U\cdot L_0 \cap U_{w_0} \cdot L_0'}$ there is a $\mathbb{P}^1$ which is the closure of an affine root subgroup orbit $U_{\eta} \cdot L_0'$. The affine root subgroup $U_{\eta}$ also acts on $L_1'$ and its orbit closure connects $L_1'$ to the lattice indexed by the coweight $\gamma' - (\mu' - \mu)$. However, $\gamma' - (\mu' - \mu) \neq \gamma$. 
In fact, the closure of the orbit $U_{\eta} \cdot (\mu', e)$ connects $(\mu', e)$ to $(\mu, w_0)$. The $T-$equivariant moment map image of the orbit closure $\overline{U_{\eta} \cdot (\mu', e)}$ is strictly bigger than the moment polytope of 
$\overline{S_e^{(\mu, e)} \cap S_{w_0}^{(\mu', e)}}$. We have arrived at a contradiction. Therefore, the intersection $S_e^{(\mu, e)} \cap S_{w_0}^{(\mu', e)}$ must be empty. 

The proof for the case of two non-translation elements is completely analogous. 

\end{proof}

Now we have a necessary and sufficient condition for non-emptiness of GGMS strata in the affine flag variety for $SL_2(\mathbb{C})$. 

\begin{corollary}
\label{SL2GGMS}

In the affine flag variety for $SL_2(\mathbb{C})$, a GGMS stratum $S_e^{\gamma_1} \cap S_{w_0}^{\gamma_2}, \gamma_1, \gamma_2 \in W_{aff}$ is nonempty if and only if $\gamma_2 \geq_{U^-} \gamma_1$  and they are either equal or can be connected by an extended torus invariant $\mathbb{P}^1$. 

\end{corollary}

\begin{proof}

On one hand, suppose $\gamma_2 \geq_{U^-} \gamma_1$  and they are either equal or can be connected by an extended torus invariant $\mathbb{P}^1$. Let $p_1$ and $p_2$ denote the $T-$fixed point(s) indexed by $\gamma_1$ and $\gamma_2$ respectively. 
If $\gamma_1 = \gamma_2$, the intersection of these two semi-infinite orbits is just the $T-$fixed point $p_1 = p_2$. If $\gamma_2 >_{U^-} \gamma_1$, then an open dense subset of the extended torus invariant $\mathbb{P}^1$ that connects $p_1$ and $p_2$ is contained in the GGMS stratum $S_e^{\gamma_1} \cap S_{w_0}^{\gamma_2}$. 

On the other hand, suppose a GGMS stratum $S_e^{\gamma_1} \cap S_{w_0}^{\gamma_2}, \gamma_1, \gamma_2 \in W_{aff}$ is nonempty. Then $\gamma_2 \geq \gamma_1$. By \Cref{semiinfiniteemptylemma} above, it is necessary that $\gamma_1 = \gamma_2$ or $\gamma_1$ and $\gamma_2$ are not both translation or both non-translation elements, which means that they can be connected by the closure of an affine root subgroup orbit by 
\Cref{rootfixedptsprop}. 

\end{proof}

In \cite{PRS}, there is an explicit description of the points in a generalized MV cycle via certain labeled alcove walks. These walks have a particularly simple form in the case of $SL_2(\mathbb{C})$, and some basic properties of generalized MV cycles in this special case follow from that. 

\begin{prop}
\label{SL2alcovewalks}

Let $G = SL_2(\mathbb{C})$. We define the positive direction on the set of alcoves and affine Weyl group elements to be the increasing direction for the $U^-$ periodic Bruhat order. 

Let $\vec{\gamma_1}$ denote the direct walk from $e$ to $\gamma_1$ across alcoves. We define a labeled folded path $p$ of type $\vec{\gamma_1}$ which ends in $\gamma_2$ as follows. If $\gamma_1 \geq_{U^-} e$, then $\vec{\gamma_1}$ is a sequence of $l(\gamma_1)$ many positive steps. Each labeled folded path $p$ is equal to $\vec{\gamma_1}$ as a walk. If $\gamma_1 \leq_{U^-} e$, then $\vec{\gamma_1}$ is a sequence of $l(\gamma_1)$ many negative steps. The walk $p$ consists of a direct negative walk from $e$ to a unique alcove $\gamma_3$ such that $e \geq_{U^-} \gamma_3 \geq_{U^-} \gamma_1$, followed by a positive fold at the hyperplane next to $\gamma_3$ in the negative direction, and a direct positive walk from $\gamma_3$ to 
$\gamma_2$. In both cases, each positive step or fold is labeled with a nonzero complex number, and each negative step is labeled with zero. 

Then points in the intersection $S = I^{\gamma_1} \cap S^{\gamma_2}_{w_0}$ in the affine flag variety are in bijection with the set of labeled folded paths $p$ of type $\vec{\gamma_1}$ which ends in $\gamma_2$. The dimension of $S$ is given by the number of distinct nonzero labels in a path $p$. 

An open neighborhood of the $T-$fixed point indexed by $\gamma_2$ in $S$ is generated by a set of orbits of distinct affine root subgroups, and these affine roots are defined by the hyperplanes at which a positive step or fold occurs.

\end{prop}

\begin{proof}

This immediately follows from \cite{PRS}, especially Theorem 7.1. 

\end{proof}

\begin{corollary}
\label{genMVSL2nonempty}
Let $G = SL_2(\mathbb{C})$. Let $S = I^{\gamma_1} \cap S^{\gamma_2}_{w_0}$ be an intersection of an Iwahori orbit and a $U^-$ orbit in the affine flag variety. If $\gamma_1 \geq_{U^-} e$, then $S$ is nonempty if and only if $\gamma_1 = \gamma_2$. 
If $\gamma_1 \leq_{U^-} e$, then $S$ is nonempty if and only if one of the following two conditions are satisfied: 
(1) $\gamma_1$ and $\gamma_2$ are equal; (2) $\gamma_2 \leq_B \gamma_1$ and they can be connected by an extended torus invariant $\mathbb{P}^1$. 

If $S$ is nonempty, it only has one irreducible component. 

\end{corollary}

\begin{proof}
This follows from \Cref{SL2alcovewalks}. 
%may need to say a few more words in the proof, especially the second if and only if...
\end{proof}

Next we compare affine Schubert varieties in the affine flag variety for $G = SL_2(\mathbb{C})$ with generalized MV cycles.

\begin{prop}
\label{IwagenMVcoro}

When $G = SL_2(\mathbb{C})$, the closure of each Iwahori orbit, $\overline{I^{\gamma_1}}, \gamma_1 \in W_{aff}$, in the affine flag variety is a generalized MV cycle. 
If $\gamma_1 \geq_{U^-} e$, then $\overline{I^{\gamma_1}} = \overline{I^{\gamma_1} \cap S_{w_0}^{\gamma_1}}$. If $\gamma_1 <_{U^-} e$, then $\overline{I^{\gamma_1}} = \overline{I^{\gamma_1} \cap S_{w_0}^{s_1\cdot \gamma_1}}$. 

\end{prop}

\begin{proof}

By \Cref{SL2alcovewalks}, the intersection of an Iwahori orbit and a $U^-$ orbit as specified in the proposition has the same dimension as the Iwahori orbit itself. Therefore their closures agree.

\end{proof}

Below we describe each generalized MV cycle as the closure of a GGMS stratum, and characterize their moment polytopes.

\begin{prop}
\label{genMVpolyprop}

Let $G = SL_2(\mathbb{C})$. Let $P$ denote the moment polytope of the closure of a nonempty intersection $S = I^{\gamma_1} \cap S^{\gamma_2}_{w_0}$. If $\gamma_1 >_{U^-} e$, then $\gamma_2 = \gamma_1$ and 
$\overline{S} = \overline{S_{w_0}^{\gamma_2} \cap S_e^{s_0 \cdot \gamma_1}}$. The moment polytope $P$ is the line interval with vertices $\gamma_2$ and $s_0 \cdot \gamma_1$. If $\gamma_1 \leq_{U^-} e$, then $\overline{S} = \overline{S_{w_0}^{\gamma_2} \cap S_e^{\gamma_1}}$. The moment polytope $P$ is the point $\gamma_1$ if $\gamma_2 = \gamma_1$, and is the line interval with vertices $\gamma_1$ and $\gamma_2$ if $\gamma_2 \neq \gamma_1$. 

The dimension of $S$ is equal to the root number of any vertex $v$ of $P$.  

\end{prop}

\begin{remark}
By \Cref{IwagenMVcoro}, \Cref{genMVpolyprop} gives a description of affine Schubert varieties in $Fl_{SL_2(\mathbb{C})}$ as the closures of GGMS strata, and characterizes their moment polytopes. 
\end{remark}

\begin{proof}

Let $J$ be the set of affine Weyl group elements that are smaller than $\gamma_1$ according to the usual Bruhat order. Let $\gamma'$ denote the smallest element in $J$ according to the $U^-$ semi-infinite Bruhat order. Note that $\gamma' = \gamma_1$ if $\gamma_1 \leq_{U^-} e$, and $\gamma' = s_0\gamma_1$ if $\gamma_1 >_{U^-} e$.  Let $P_2$ denote the line interval with vertices $\gamma_2$ and $\gamma'$, then $P \subseteq P_2$. We would like to show that $P = P_2$. 

By \Cref{SL2alcovewalks}, the dimension of a generalized MV cycle is given by $m$, the total number of positive steps and folds in its corresponding labeled alcove walk. Let $J'$ denote the set of $m$ distinct affine simple reflections across the hyperplanes represented by positive steps and folds in the alcove walk. Then the lattice points in $P_2$ are precisely the images of $\gamma_2$ under the affine simple reflections in $J'$. Let $J''$ denote the set of $T-$fixed points whose moment map images are these $m$ lattice points in $P_2$. Geometrically, for each $T-$fixed point $p$ in $J''$, there is an affine root subgroup whose orbit closure connects $p'$, the $T-$fixed point indexed by $\gamma_2$, and $p$. These $m$ affine root subgroup orbits containing $p'$ generate an open neighborhood of $p'$ in $S$. Therefore the $T-$fixed point indexed by $\gamma'$ belongs to this generalized MV cycle, and $P = P_2$. As a result, the dimension of $S$ is the root number of the vertex $\gamma_2$. By symmetry, the same is true for the other vertex $\gamma'$. 

As for the GGMS strata part, by \Cref{GGMSirreprop}, $\overline{S}$ is contained in the closure of a GGMS stratum $S'$ with the same moment polytope. By \Cref{localdimprop}, the dimension of any irreducible component of $S'$ is less than or equal to the root number a vertex of $P$, which is equal to the dimension of $S$. Therefore, $\overline{S}$ and $\overline{S'}$ must be equal. 

\end{proof}

The next theorem precisely describes the special fiber limits of MV cycles via generalized MV cycles and GGMS strata. It also characterizes the moment polytopes of the irreducible components for the special fiber limits. 

\begin{theorem}
\label{MVSL2thrm}

Let $G = SL_2(\mathbb{C})$. Let $S = \overline{Gr^{\lambda} \cap S_{w_0}^{\mu}} = \overline{S_{e}^{- \lambda} \cap S_{w_0}^{\mu}}$, where $\lambda \in X_*(T)^+, \mu \in X_*(T)$, be an MV cycle, and $\tilde{S}$ be its special fiber limit in the affine flag variety. If $\mu = -\lambda$, then $\tilde{S}$ is the $T-$fixed point indexed by $(-\lambda, e)$. Otherwise $\tilde{S}$ has exactly two irreducible components. Each irreducible component of $\tilde{S}$ is a generalized MV cycle and the closure of a GGMS stratum in the affine flag variety. 

More specifically, consider the case of $\mu \neq -\lambda$.  If $\lambda = \mu$, then $S$ is the closure of the $G(\mathcal{O})$ orbit $Gr^{\lambda}$. The two irreducible components of $\tilde{S}$ are the closures of two Iwahori orbits 
$\overline{I^{(\lambda, e)}}= \overline{I^{(\lambda, e)} \cap S_{w_0}^{(\lambda, e)}}$, and $\overline{I^{(-\lambda, e)}} = \overline{I^{(-\lambda, e)} \cap S_{w_0}^{(\lambda, w_0)}}$, both of which are special cases of generalized MV cycles. 
If $-\lambda < \mu < \lambda$, the two irreducible components are the generalized MV cycles $\overline{I^{(-\lambda, e)} \cap S_{w_0}^{(\mu, w_0)}}$ and 
$\overline{I^{(-\lambda + \alpha, w_0)} \cap S_{w_0}^{(\mu, e)}}$. 

In both cases, the irreducible component containing $(\mu, e)$ is equal to the closure of the GGMS stratum $S_e^{(-\lambda + \alpha, w_0)} \cap S_{w_0}^{(\mu, e)}$, and the moment polytope of this component is the line interval with vertices $(-\lambda + \alpha, w_0)$ and $(\mu,e)$. The irreducible component containing $(\mu, w_0)$ is equal to the closure of the GGMS stratum $S_e^{(-\lambda, e)} \cap S_{w_0}^{(\mu, w_0)}$, and the moment polytope of this component is the line interval with vertices $(-\lambda, e)$ and $(\mu, w_0)$. 

\end{theorem}

\begin{proof}

The case of $\mu = -\lambda$ is obvious. We focus on the case of $\mu \neq -\lambda$ in this proof. 

By \Cref{basicMVlemma}, $\tilde{S}$ is contained in a finite union of intersections of Iwahori orbits and $U^-$ orbits. Among these intersections, only two have the dimension of $S$ by \Cref{SL2alcovewalks}. If $\lambda = \mu$, then the original MV cycle is equal to the closure of a $G(\mathcal{O})$ orbit. Its special fiber limit is described in \cite{Zhu}. The description of the closures of these two Iwahori orbits as generalized MV cycles follows from \Cref{IwagenMVcoro}. If $-\lambda < \mu < \lambda$, then the two relevant generalized MV cycles of the right dimension are the closures of $I^{(-\lambda, e)} \cap S_{w_0}^{(\mu, w_0)}$ and $I^{(-\lambda + \alpha, w_0)} \cap S_{w_0}^{(\mu, e)}$ respectively. Both of these two generalized MV cycles must be an irreducible component of $\tilde{S}$ as both of them contain special fiber limits of $T-$fixed points in the original MV cycle. The descriptions in terms of GGMS strata and moment polytopes follow from \Cref{genMVpolyprop}.

\end{proof}

Next we present an explicit example of the central degeneration of an MV cycle in the affine Grassmannian for $SL_2(\mathbb{C})$. 

\begin{example}

Let $\alpha$ be the simple coroot for $G = SL_2(\mathbb{C})$. Consider the MV cycle $\overline{Gr^{2\alpha} \cap S_{w_0}^{\alpha}}$. 
%This also happens to be the opposite Iwahori orbit indexed by $- 2 \alpha \in X_*(T)$. 

This degeneration is illustrated in terms of moment polytopes below: 

\includegraphics[scale = 0.8]{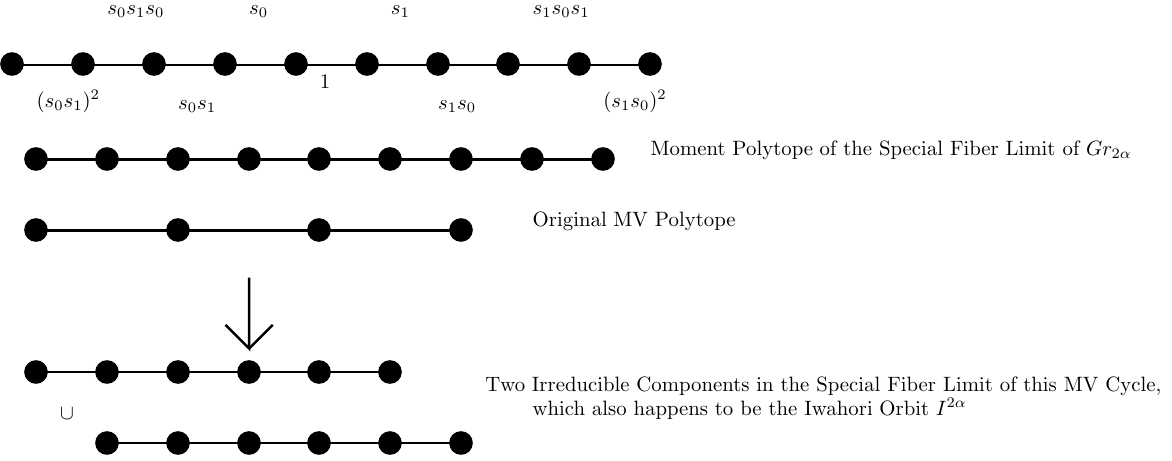}

The first line of dots is the alcove diagram for $SL_2(\mathbb{C})$. The second line gives the moment polytope of the special fiber limit of the closure of the entire $G(\mathcal{O})$ orbit $Gr^{2 \alpha}$. The moment map image of each $T-$fixed point lies in the interior of a unique alcove. The third line is the moment polytope of our MV cycle in the general fiber $Gr \times \{ e\}$. Each $T-$fixed point of our original MV polytope lies in the interior of an alcove that corresponds to a translation element in the affine Weyl group. The last two lines illustrate the moment polytopes of the two irreducible components for the special fiber limit of the original MV cycle.

\end{example}

\subsection{General $SL_n(\mathbb{C})$ case}

In this subsection we would like to shed some light on the special fiber limits of MV cycles for general $SL_n({\mathbb{C}})$ by relating it with the special fiber limits for the $SL_2(\mathbb{C})$ case. 

First we remark that unlike in the $SL_2(\mathbb{C})$ case, not every irreducible component in the special fiber limit of an MV cycle for general $G$ in type A is a generalized MV cycle for $G$. 

\begin{example}

Consider one of the MV cycles in $Gr^{\alpha + \beta} \cap S_{w_0}^{e}$ for $G = SL_3(\mathbb{C})$. It is two-dimensional and isomorphic to 
$\mathbb{P}^2$. Its special fiber limit has three irreducible components, and one of them is contained in the generalized MV cycle 
$\overline{I^{s_1s_2s_1} \cap S_{w_0}^e}$, which is also the three-dimensional $G/B$ bundle above $t^e$ in the affine Grassmannian. We denote this component $V$.  

We claim that $V$ is not a generalized MV cycle. To see this, note that $V$ has four $T-$fixed points. Among them, the longest element in $W_{aff}$ is $s_1s_2s_1$, and the biggest element according to the $U^-$ periodic Bruhat ordering is $e$. So if 
$V$ were a generalized MV cycle, it has to be the closure of an irreducible component of $I^{s_1s_2s_1} \cap S_{w_0}^e$. However, we already know that the closure of $I^{s_1s_2s_1} \cap S_{w_0}^e$ is isomorphic to $G/B$, whose dimension is higher than that of $V$. Therefore, $V$ cannot be a generalized MV cycle.

\end{example}

\begin{conj}

If an MV cycle is the only irreducible component in the closure of the intersection of a $G(\mathcal{O})$ orbit and a $U^-$ orbit in the affine Grassmannian $Gr_{SL_n(\mathbb{C})}$, then each irreducible component of the special fiber limit of this MV cycle would be a generalized MV cycle in the affine flag variety. 

In general, an irreducible component for the special fiber limit of an MV cycle for $G = SL_n(\mathbb{C})$ is isomorphic to a type A generalized MV cycle for $GL_m(\mathbb{C})$ for some $m \in \mathbb{Z}^{> 0}$.

\end{conj}

In \cite{BrGai} and \cite{Kam2}, a crystal structure was defined on the set of MV cycles and MV polytopes of a group $G$ by restricting/projecting each MV cycle to an MV cycle for a rank one Levi subgroup. For the remainder of this subsection, we will also use this Levi restriction map on the affine Grassmannian and affine flag variety to discover properties about special fiber limits of MV cycles in general $SL_n(\mathbb{C})$ case. Our results can be easily extended to all MV cycles for $GL_n(\mathbb{C})$.  

\begin{prop}
\label{restriinterprop}

Let $P^J \subseteq B_w$ be a parabolic subgroup with unipotent radical $N^J$ and Levi factor $G_J$, where $J$ is the associated subset of the set of vertices of the Dynkin diagram. 
Given $u_1, u_2 \in W_J, \mu_1, \mu_2 \in X_*(T), \gamma_1, \gamma_2 \in W_{aff}$, $r_P(S_{wu_1}^{\mu_1} \cap S_{wu_2}^{\mu_2} \subset Gr_G) = S_{u_1, J}^{\mu_1} \cap S_{u_2, J}^{\mu_2} \subset Gr_{G_J}$, and 
$r_P(S_{wu_1}^{\gamma_1} \cap S_{wu_2}^{\gamma_2} \subset Fl_G) = S_{u_1, J}^{\gamma_1} \cap S_{u_2, J}^{\gamma_2} \subset Fl_{G_J}$. 

\end{prop}

\begin{proof}

The claim follows from the arguments in the proof of the analogous statement Prop. 4.3 in \cite{Kam2}.

\end{proof}

The following result about GGMS strata immediately follows. 

\begin{corollary}
\label{GGMSrestricoro}

Let $G = SL_n(\mathbb{C})$. Given a GGMS stratum $S = \cap_{w \in W}S_w^{\eta_w}$ in the affine Grassmannian(where $\eta_w \in X_*(T)$) or in the affine flag variety (where $\eta_w \in W_{aff}$) with moment polytope $P$. Consider a pair of adjacent vertices of $P$, $\eta_{w_1}$ and $\eta_{w_2}$, such that $w_2 = w_1s_j$ for some simple reflection $s_j$. Let $J = \{j\}$, $P^J \supseteq B_{w_1}$ be the associated parabolic subgroup and $G_J$ be its Levi factor. Then the image of $S$ under $r_P$ is contained in the GGMS stratum $S_{e, J}^{\eta_{w_1}} \cap S_{s_j, J}^{\eta_{w_2}}$ in $Gr_{G_J}$ or $Fl_{G_J}$.  

\end{corollary}

\begin{prop}
\label{GGMSflagprop}

Given a nonempty GGMS stratum $S = \cap_{w \in W} S_w^{(\gamma_w)}, \gamma_w \in W_{aff}$ in the affine flag variety of $G$, let $P$ denote its moment polytope. Then every pair of adjacent vertices of $P$ in the alcove lattice must be connected by an orbit of an affine root subgroup of $G(\mathcal{K})$. 

\end{prop}

\begin{proof}

By \Cref{GGMSrestricoro}, there exists a rank one Levi group $G_J$ such that the image of $S$ under the Levi restriction map $r_P$ is contained in a GGMS stratum $S'$ of $Fl_{G_J}$. One connected component of $S'$ is isomorphic to a GGMS stratum in the affine flag variety of $SL_2(\mathbb{C})$. Since $S$ is nonempty, $S'$ is also nonempty. By \Cref{SL2GGMS}, for $S'$ to be nonempty, it is necessary that $\gamma_{w_1}$ and $\gamma_{w_2}$ can be connected by an extended torus invariant $\mathbb{P}^1$.

\end{proof}

\begin{prop}
\label{SL_2MVsubprop}

Let $G = SL_n(\mathbb{C})$. Given an MV cycle $S = \overline{\cap_{w \in W} S_w^{(\mu_w)}}, \mu_w \in X_*(T)$ with MV polytope $P$, consider any pair of adjacent vertices of $P$, $\mu_{w_1} > \mu_{w_2}$. There exists an irreducible subscheme $S'$ of $S$ containing the two $T-$fixed points indexed by $\mu_{w_1}$ and $\mu_{w_2}$ such that $S'$ is isomorphic to the MV cycle $\overline{S_e^0 \cap S_{s_1}^{(\mu_{w_1} - \mu_{w_2})}}$ in the affine Grassmannian for $SL_2(\mathbb{C})$. 

\end{prop}

\begin{proof}

By \Cref{GGMSrestricoro} and Theorem 4.6 of \cite{Kam2} or Proposition 3.1 of \cite{BrGai}, there exists a rank one Levi group $G_J$ such that the image of $S$ under the Levi restriction map $r_P$ is equal to the closure of an MV cycle $S_J$ in $Gr_{G_J}$. One connected component of $S_J$ is isomorphic to the MV cycle $S' = \overline{S_e^0 \cap S_{s_1}^{(\mu_{w_1} - \mu_{w_2})}}$ in the affine Grassmannian for $SL_2(\mathbb{C})$. Since $r_P$ is a retraction map, $S'$ is isomorphic to a subscheme of $S$ containing the two $T-$fixed points indexed by $\mu_{w_1}$ and $\mu_{w_2}$. 

\end{proof}

\begin{theorem}
\label{MVlimitpolythrm}

Let $G = SL_n(\mathbb{C})$. Given an MV cycle $S = \overline{\cap_{w \in W}S_w^{\mu_w}}$ with special fiber limit $\tilde{S}$. Let $P$ denote the moment polytope of $S$ and $\tilde{S}$ with vertices $(\mu_w, e), w \in W$. Given any irreducible component $S'$ of $\tilde{S}$, its moment polytope $P'$ is a Pseudo-Weyl polytope in $P$. The following two properties must be satisfied by $P'$: (1) every pair of adjacent vertices of $P'$ can be connected by an extended torus invariant $\mathbb{P}^1$; (2) the root number of any vertex $v$ of $P'$, $n_{P'_v}$, is bigger than or equal to the dimension of $S$.  

In $\tilde{S}$, each extremal $T-$fixed point indexed by $(\mu_w, e)$ lies in a unique irreducible component $S'_w$. The moment polytope $P'_w$ of $S'_w$ satisfies the following additional property: (3) the set of adjacent vertices of $(\mu_w, e)$ in $P'_w$, denoted by $C_1$, is in bijection with the set of adjacent vertices of $(\mu_w, e)$ in $P$, denoted by $C_2$. Given a vertex $(\eta, e) \in W_{aff}$ in $C_2$, the corresponding element $f_{\eta}$ in $C_1$ is the affine Weyl group element adjacent to $(\eta, e)$ on the edge connecting $(\mu_w, e)$ and $(\eta, e)$ in $P$. 

\end{theorem}

\begin{proof}

Given any irreducible component $S'$ of $\tilde{S}$, by \Cref{basicMVlemma} there exists a GGMS stratum $\hat{S}$ whose closure contains $S'$ and shares the same moment polytope $P'$ with $S'$. Then property (1) immediately follows from \Cref{GGMSflagprop}. 

As for property (2), note that by \Cref{localdimprop}, the root number of any vertex $v$ of $P'$ is bigger than or equal to the dimension of any irreducible component of $\overline{\hat{S}}$ that contains the $T-$fixed point corresponding to $v$, which is in turn bigger than $\dim(S') = \dim(S)$. 

Given the extremal $T-$fixed point indexed by $\mu_w \in X_*(T)$ of $S$, its open neighborhood in $S$ degenerates to an irreducible open neighborhood of $(\mu_w, e)$ in $\tilde{S}$ by \Cref{rootdegcoro}. Therefore, the $T-$fixed point indexed by $(\mu_w, e)$ lies in a unique irreducible component of $\tilde{S}$. 

For statement (3), consider an adjacent vertex $(\eta, e)$ of $(\mu_w, e)$ in $C_2$. By \Cref{SL_2MVsubprop}, there exists a subscheme $V_{\eta}$ containing the two $T-$fixed points indexed by $\eta$ and $(\mu_w, e)$ such that $V_{\eta}$ is isomorphic to an MV cycle of the affine Grassmannian for $SL_2(\mathbb{C})$. Therefore, the special fiber limit $\tilde{V_{\eta}}$ of $V_{\eta}$, which is isomorphic to that of an MV cycle for $SL_2(\mathbb{C})$, is contained in $\tilde{S}$. By \Cref{MVSL2thrm}, $f_{\eta}$ as specified in the statement of the theorem is a $T-$fixed point in $\tilde{V_{\eta}}$. By \Cref{GGMSflagprop} and \Cref{rootfixedptsprop}, $(\eta, e)$ does not lie in $P'_w$, as it is also a translation element and therefore cannot be connected with $(\mu_w, e)$ via an extended torus invariant $\mathbb{P}^1$. Therefore, $f_{\eta}$ is one adjacent vertex of $(\mu_w, e)$ in $P'_w$.

\end{proof}

In many familiar examples, the irreducible components of the special fiber limit are in one-one correspondence with the vertices of the MV polytope.

\begin{example}

Let $G = GL_n(\mathbb{C})$. The closure of any $G(\mathcal{O})$ orbit $Gr^{\lambda}$ for minuscule $\lambda$ is an ordinary Grassmannian. The explicit equations for central degeneration in this case are worked out in \cite{Go}. 

In the special case where $Gr^{\lambda} \cong \mathbb{P}^n$, all the MV cycles are $\mathbb{P}^m$ for some positive integer $m \leq n$. Given an MV cycle which is isomorphic to $\mathbb{P}^k$, its special fiber limits have $k+1$ irreducible components \cite{Go}, one for each $T-$fixed point in $\mathbb{P}^k$. 

\end{example}

In the next part, we present our main example. In particular, with some techniques from symplectic geometry, we demonstrate the existence of an irreducible component in the special fiber limit of an MV cycle for $SL_3(\mathbb{C})$ that does not contain any of the extremal $T-$fixed points.

\subsection{Main example via fibers of moment maps}

In this subsection, we focus on the case of $G = SL_3(\mathbb{C})$, where $\alpha$ and $\beta$ are the two simple coroots. We will describe the degenerations of all the MV cycles in the closure of the $G(\mathcal{O})$ orbit $\overline{Gr^{\alpha + \beta}}$. These MV cycles give a canonical basis of the adjoint representation of $PGL_3(\mathbb{C})$ (or $SL_3(\mathbb{C})$).  Each of them is the closure of an irreducible component of the intersection $Gr^{\alpha + \beta} \cap S_{w_0}^{\mu}$, where $\mu$ is a coweight that lies in the convex hull of $w \cdot (\alpha + \beta), w \in W$. 

In the diagrams below, the lattice formed by the hyperplanes is the coweight lattice of $G$, and the alcoves are labeled by affine Weyl group elements. Various convex polytopes lie in the weight lattice, and their vertices all lie in the interior of alcoves. The three long arrows represent finite roots, and the shaded cone is the set of dominant coweights up to a choice. 

First, when $\mu = - (\alpha + \beta)$, the MV cycle is the $T-$fixed point $t^{\mu}$, and it degenerates to the $T-$fixed point indexed by $(\mu, e)$ in the affine flag variety. 

Then we look at the MV cycles which are the closures of the orbits of $T$ or its subgroups. Consider a flat degeneration of a toric variety. By \cite{Stu}, the moment polytopes of the irreducible components for the special fiber gives rise to a regular subdivision of the moment polytope of the general fiber. Also, distinct irreducible components have distinct moment polytopes. 

When $\mu = - \alpha$ or $- \beta$, this MV cycle a $T \times \mathbb{C}^*$ invariant $\mathbb{P}^1$. In particular, it is a toric variety for a one-dimensional subtorus in $T$. Its MV polytope, which is a line interval, only has three lattice points, so there is only one possible regular subdivision. The two subintervals in the MV polytope are the moment polytopes of the two irreducible components in the special fiber limit. Overall, this is an example of a $\mathbb{P}^1$ degenerates to two copies of $\mathbb{P}^1$ intersecting at a point. 

When $\mu = e$, each of the two MV cycles is isomorphic to $\mathbb{P}^2$, which is a toric variety for the maximal torus $T \subset G$. The MV polytope is the big triangle below, and there is only one possible regular subdivision it, as shown below. The green, yellow and purple convex polytopes in the subdivision are the moment polytopes of the three irreducible components in the special fiber limit of the MV cycle. 

\includegraphics[scale = 0.6]{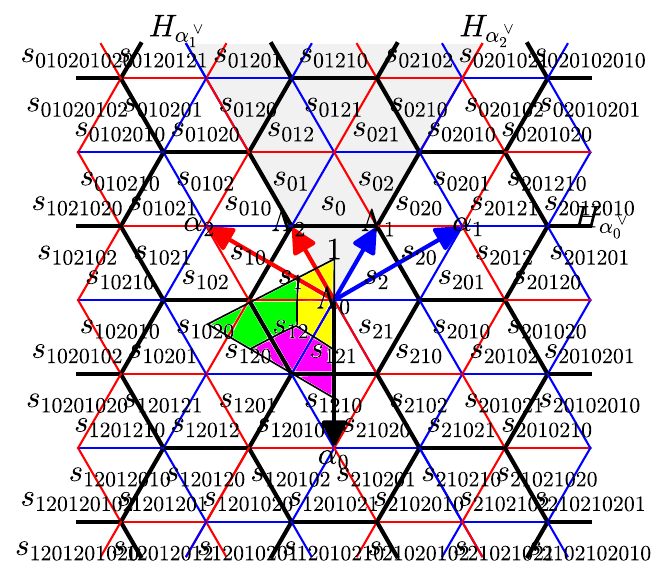} 

Note that this MV cycle is isomorphic to the minuscule $G(\mathcal{O})$ orbit $Gr^{(1, 0, 0)} \cong \mathbb{P}^2$ in the affine Grassmannian for $GL_3(\mathbb{C})$. By \cite{Go}, the special fiber limit of $Gr^{(1, 0, 0)} \subset Gr_{GL_3(\mathbb{C})}$ also has three irreducible components,  and their moment polytopes are isomorphic to the three colored moment polytopes shown above. 

Also, let's not forget about the MV cycle that is equal to the closure of the entire $G(\mathcal{O})$ orbit. The MV polytope is the big hexagon below. As we have explained earlier, by \cite{Zhu}, the special fiber limit has $|W|$ many irreducible components, each of which is the closure of a four-dimensional Iwahori orbit in the affine flag variety. Their moment polytopes are illustrated as six single-colored (red, orange, green, yellow, purple, light blue) convex polytopes below. 

\includegraphics[scale = 0.6]{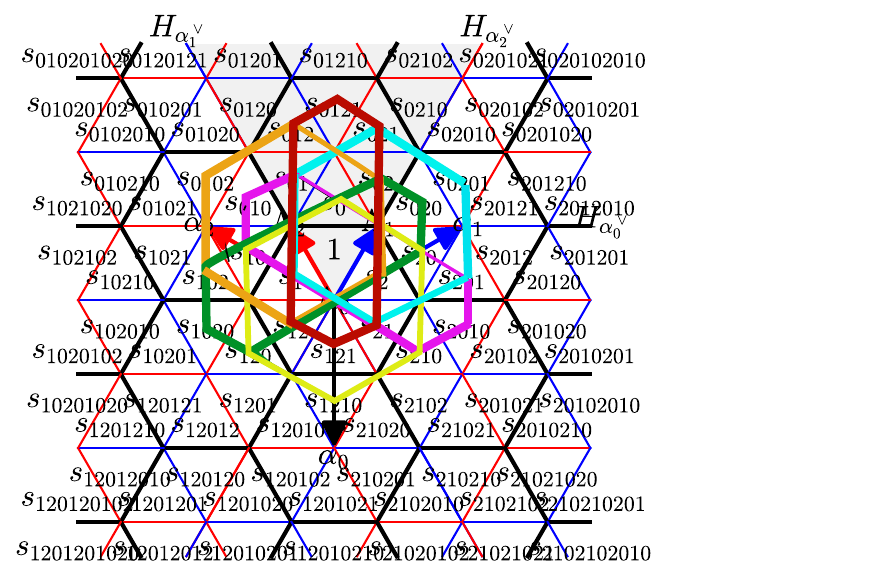}

Finally, when $\mu = \alpha$ (or $\beta$), the MV cycle is three-dimensional, and the MV polytope is the big trapezoid below. In all the cases above, the irreducible components in the special fiber limit of the given MV cycle are in bijection with certain convex polytopes containing vertices of the MV polytope. In this case, we would like to demonstrate the existence of an irreducible component in the special fiber limit whose moment polytope does not contain any vertex of the MV polytope. 

Each vertex of the MV polytope is contained in a unique irreducible component for the special fiber limit, and their moment polytopes are illustrated as single-colored convex polytopes below. The four extremal moment polytopes are denoted as $P_{green}, P_{blue}, P_{purple}$ and $P_{orange}$ according to their colors. Similarly, the corresponding irreducible components are denoted as $C_{green}, C_{blue}, C_{purple}$ and $C_{orange}$.

\includegraphics[scale = 0.6]{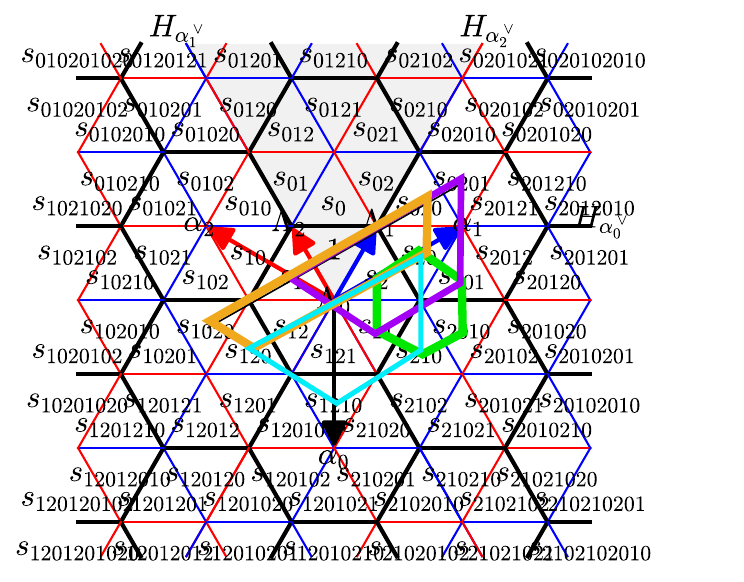}

We are going to show that there exists an irreducible component whose moment polytope does not contain any of the vertices of the original MV polytope. In fact, there are exactly five irreducible components in the special fiber limit, and we illustrate their moment polytopes as the single colored convex polytopes below. 

\includegraphics[scale = 0.8]{"Decomposition_trapezoid".pdf} 

As explained in section 3, we can define moment maps $\Phi_{T_{\mathbb{R}}}$ and $\Phi_{(T \times \mathbb{C}^*)_{\mathbb{R}}}$ for the actions of $T \subset G$ and the extended torus 
$T \times \mathbb{C}^*$ on the affine Grassmannian and affine flag variety. There is a natural projection map $proj$ between their moment map images, and the diagram below commutes.

\xymatrixrowsep{10mm} 
\xymatrixcolsep{2pc}
\xymatrix{
 & X \ar[rd]^{\Phi_{(T \times \mathbb{C}^*)_{\mathbb{R}}}}  \ar[ld]_{\Phi_{T_{\mathbb{R}}}} \\
\mathbb{R}\delta \oplus \mathfrak{t}^*_{\mathbb{R}} \ar[rr]^{proj} & & \mathfrak{t}^*_{\mathbb{R}} 
}

Let $K$ denote the maximal compact subgroup $SU_3(\mathbb{C})$ of $G$. Recall that the affine Grassmannian $Gr_G = L_{alg}G/L_{alg}^+G$ is homeomorphic to the based loop group of $K$, $\Omega_K = L_{alg}K/K$ by \cite{PS}. By Lemma 3.3 in \cite{HHJM}, when restricted to the closure of each Bruhat cell of $\Omega_K$, each fiber of the moment map $\Phi_{(T \times \mathbb{C}^*)_{\mathbb{R}}}$ is path connected. Since this MV cycle $S$ is also an affine Schubert variety in $Gr_G$, each fiber of $\Phi_{(T \times \mathbb{C}^*)_{\mathbb{R}}}$ when restricted to $S$ is path connected as well. Through an explicit calculation of the $T$ and $T \times \mathbb{C}^*$ equivariant moment polytopes of $S$, we conclude that each fiber of the natural projection $proj$ is also path connected. By the commutative diagram, each fiber of $\Phi_{T_{\mathbb{R}}}$ when restricted to $S$ is path connected. Since this flat degeneration is $T-$equivariant, for each point in the MV polytope $P$, its fiber in $S$ degenerates to its fiber in the limit $\tilde{S}$. Therefore, each fiber of $\Phi_{T_{\mathbb{R}}}$ when restricted to $\tilde{S} \subset Fl_G$ is also path connected.

Let $P_{black}$ denote the black-colored internal hexagon in the diagram above. It intersects nontrivially with all the other four single-colored extremal convex polygons $P_{green}, P_{blue}, P_{purple}$ and $P_{orange}$. Let $H_g, H_b, H_p, H_o$ be the intersections of $P_{green}, P_{blue}, P_{purple}, P_{orange}$ with $P_{black}$ respectively. Let $p_1$ and $p_2$ be an interior point in $H_b \cap H_p$ and $H_o \cap H_p$ with fibers $f_1$ and $f_2$ respectively. Suppose $f_1$ is contained in the union of $C_{blue}$ and $C_{purple}$. Since $f_1$ is path connected, $f_1 \cap C_{blue} \cap C_{purple} \neq \emptyset$. This is true for every interior point of $H_b \cap H_p$. Let $h$ denote the unique interior lattice point of $P_{black}$. Note that $h$ is not the moment map image of any $T-$fixed point as it does not lie in the interior of any alcove. On the other hand, the moment polytope of the intersection $C_{blue} \cap C_{purple}$ must contain $h$ as a vertex as we can choose $p_1$ to be arbitrarily close to $h$. This is a contradiction. Therefore, there exists at least another irreducible component $C_{black}$ such that $f_1 \cap C_{black} \neq \emptyset$. The same is true for $f_2$. As a result, the only possible new moment polytope is indeed the Pseudo-Weyl polytope $P_{black}$. This internal irreducible component is contained in the closure of a GGMS stratum, which is equal to the three-dimensional $G/B$ bundle above identity point in the affine Grassmannian, as well as the closure of the three-dimensional Schubert cell $I^{s_1s_2s_1}$. Therefore, $P_{black}$ is the moment polytope of the unique irreducible component $\overline{I^{s_1s_2s_1}}$. 

Apart from these five irreducible components, there is no other irreducible component. This is because there is no other Pseudo-Weyl polytope in $P$ that satisfies the properties specified in \Cref{MVlimitpolythrm}.  

\begin{remark}

To compute the special irreducible component above, we used the fact that this MV cycle happens to be an affine Schubert variety. For a related study of the images of affine Schubert varieties under the map 
$\Omega K \rightarrow LK \rightarrow LK/T$ from the viewpoint of affine Schubert calculus, see \cite{Lam}. 

\end{remark}

\section{Connections with Demazure modules and Affine Deligne-Lusztig varieties}

MV cycles for $G$ give a basis of the highest weight representations of $G^{\vee}$. They arise as we consider the weight functor \cite{MV} and restrict elements of $\mathcal{P}_{G(\mathcal{O})}(Gr)$ to $U^-$ orbits in $Gr$. Similarly, generalized MV cycles arise as we consider the weight functor \cite{AB} on $\mathcal{P}_I(Fl)$. Affine Deligne-Lusztig variety \cite{Rap} is a generalization of Deligne and Lusztig's classical construction \cite{DL}, and is very important for the study of Shimura varieties in number theory.  In \cite{GHKR}, the question of non-emptiness and dimension of affine Deligne-Lusztig varieties in the affine flag variety is reduced to questions related to generalized MV cycles. By \Cref{basicMVlemma}, the special fiber limits of MV cycles are contained in a union of generalized MV cycles. We expect to use the limits of central degeneration to get a better understanding of affine Deligne Lusztig varieties. When we project a generalized MV cycle in $Fl$ to $Gr$, we get an Iwahori MV cycle. Iwahori MV cycles in $Gr$ are closely related to the combinatorics of Demazure modules for the Langlands dual group, as we will see. 
 
In this section we will first prove that the intersections of Iwahori orbits and $U^-$ orbits in the affine Grassamnnian are equi-dimensional, and give an explicit dimension formula, by adopting methods developed in \cite{MV}. Combining with Theorem 11.7 in \cite{Sch}, we prove that the numbers of Iwahori MV cycles in the intersections of Iwahori orbits and $U^-$ orbits in $Gr$ are equal to the corresponding weight multiplicities in the Demazure modules for the Langlands dual group. Then we will proceed to discuss some dimension bounds for the intersections of Iwahori orbits and $U^-$ orbits in the affine flag variety. 

Let $\lambda$ denote a coweight of $G$ and $\lambda_{dom}$ denote the dominant coweight associated with $\lambda$. In the affine Grassmannian, the $G-$orbit of $t^{\lambda}$ is the partial flag variety 
$G/P_{\lambda} = G/P_{\lambda_{dom}}$, where $P_{\lambda}$ denotes the parabolic subgroup of $G$ with a Levi factor associated with the roots $\alpha$ such that $\lambda (\alpha) = 0$. Each $G(\mathcal{O})$-orbit $Gr^{\lambda}$ in the affine Grassmannian $Gr$ is a vector bundle over a partial flag variety $G/P_{\lambda}$. The vector bundle projection map is given by 

\xymatrix{
Gr^{\lambda} \cong G(\mathcal{O}) \cdot t^{\lambda} \ar[r]^{ev_0} & G/P_{\lambda_{dom}} \cong G \cdot t^{\lambda}. }
 
The fibers are isomorphic to $I_1 \cdot t^{\lambda_{dom}}$ as vector spaces, where $I_1$ is the subgroup of $I$ that is the pre-image of the identity element under the map $ev_0$. The dimension of each fiber is $2 \cdot ht(\lambda_{dom}) - \dim(G/P_{\lambda_{dom}})$.

Let $W$ denote the Weyl group of $G$. The partial flag variety 
$G/P_{\lambda}$ has a cell decomposition indexed by elements of the coset $\tilde{W} = W/W_J$, where $W_J$ is the subgroup generated by permutations of the simple roots associated with $P_{\lambda}$. Let $X_w$ denote the Schubert cell corresponding to $w \in \tilde{W}$, and $\overline{X_w}$ denote its closure. For a unique $w$ in the coset $\tilde{W}$, $\lambda = w \cdot \lambda_{dom}$. 
Then the Iwahori orbit $I^{\lambda}$ in the affine Grassmannian $Gr$ is the pre-image of the open Schubert cell $X_w$ under the map $ev_0$ above, and is a vector bundle over $X_w$. The dimension of the Iwahori orbit $I^{\lambda}$ in the affine Grassmannian is $\dim(X_{w})+ 2 \cdot \height(\lambda_{dom}) - \dim(G/P_{\lambda_{dom}})$.

Let's first calculate the intersection of Iwahori orbits with $U^-$-orbits in the affine Grassmannian using the techniques developed in \cite{MV}. 
%These are the same as the intersections of Iwahori orbits with the open MV cycles in the $G(\mathcal{O})$ orbits. 

\begin{theorem}
\label{IwaMVthrm}

Let $G$ be a connected reductive algebraic group of any type. Let $\lambda$ and $\mu$ be coweights of $G$, and let $\lambda_{dom}$ be the dominant coweight associated with $\lambda$. Let $\tilde{W} = W/W_J$ denote the quotient of the finite Weyl group associated with the partial flag variety 
$G/P_{\lambda_{dom}}$. Let $\lambda = w \cdot \lambda_{dom}$ for a unique $w \in \tilde{W}$, and $X_w$ be the Schubert cell for $w \in \tilde{W}$.

The intersection of the $U^-$ orbit $S_{w_0}^{\mu}$ with the Iwahori orbit $I^{\lambda}$ is \emph{equidimensional} and of dimension

$$\height(\lambda_{dom} + \mu) - \dim(G/P_{\lambda_{dom}}) + \dim (X_w)$$ 
when $\lambda \leq \mu \leq \lambda_{dom}$, and is empty otherwise. 

\end{theorem}

\begin{proof}

Consider the two subgroups $I$ and $U^-_{\mathcal{O}} = U^- \cap G(\mathcal{O})$. Given a $T-$fixed point $t^{\gamma}$, there is an inclusion
$U^-_{\mathcal{O}} \cdot t^{\gamma} \hookrightarrow S_{w_0}^{\gamma} \cap Gr^{\gamma} = U^- \cdot t^{\gamma} \cap Gr^{\gamma}$. 

First consider the case when $\mu = \lambda_{dom}$. 

Let $J^- = ev_0^{-1}(N^-)$, where $N^-$ is the unipotent radical of the opposite Borel $B^-$ in $G$. We have the equalities $S_{w_0}^{\lambda_{dom}} \cap Gr^{\lambda} = U^-_{\mathcal{O}} \cdot t^{\lambda_{dom}} \cap Gr^{\lambda} = J^- \cdot t^{\lambda_{dom}}$. This intersection is the pre-image of the open opposite Schubert cell $X^-_{e}$ under the map $ev_0$, and is a vector bundle over $X^-_{e}$. 

We also have the equality $S_{w_0}^{\lambda_{dom}} \cap I^{\lambda} = J^- \cdot t^{\lambda_{dom}} \cap I^{\lambda}$. 
This intersection is the pre-image of the open Richardson variety $X^-_{e} \cap X_w$ under the map $ev_0$. The open Richardson variety 
$X^-_{e} \cap X_w$ has dimension $l(w), w \in \tilde{W}$ and is dense in $X_w$. Therefore,  $S_{w_0}^{\lambda_{dom}} \cap I^{\lambda}$ is dense in the Iwahori orbit $I^{\lambda}$ and only has one irreducible component. 
The dimension of the intersection is the same as the dimension of the Iwahori orbit $I^{\lambda}$ itself, namely $2 \cdot \height(\lambda_{dom}) - \dim(G/P_{\lambda_{dom}}) + \dim(X_w)$. 

On the other hand, let's consider the case $\mu = \lambda$.

Claim: $S_{w_0}^{\lambda} \cap Gr^{\lambda} = U^-_{\mathcal{O}} \cdot t^{\lambda} \cap Gr^{\lambda}$.

Proof of Claim:

The group $U^-$ can be written as an infinite product of affine root subgroups $U_{\alpha}, \alpha$ being an affine root of the form 
$-\alpha_0 + j \delta$, where $\alpha_0$ is a positive finite root. If $j \geq 0$, $U_{\alpha}$ is a subgroup of $U^-_{\mathcal{O}}$; if $j < 0$, $U_{\alpha} \cdot t^{\lambda} \cap Gr^{\lambda}$ is equal to the $T-$fixed point $t^{\lambda}$. Therefore $S_{w_0}^{\lambda} \cap I^{\lambda} = U^-_{\mathcal{O}} \cdot t^{\lambda} \cap I^{\lambda}$.

We know that $G(\mathcal{O})$ acts on $Gr^{\lambda}$ transitively. Therefore, the tangent space at $t^{\lambda}$ in $Gr^{\lambda}$ is isomorphic to the following quotient of the loop subalgebra $L^+{\mathfrak{g}} = \oplus_{k \geq 0} \mathfrak{g} t^k$, $ L^+{\mathfrak{g}} / Z$, 
where $Z = \{ f \in L^+_{\mathfrak{g}} | f \cdot t^{\lambda} = t^{\lambda} \cdot h~\text{for~some}~h \in L^+_{\mathfrak{g}} \}.$ In fact, since each Iwahori MV cycle is invariant under the action of the extended torus, its tangent space at a torus fixed point is a direct sum of the subspaces generated by affine roots. 

There is a decomposition of the Lie algebra 
$\mathfrak{g} = \mathfrak{t} \oplus (\oplus_{\alpha_0} (\mathfrak{g}_{\alpha_0} \oplus \mathfrak{g}_{-\alpha_0}))$, 
where $\alpha_0$ ranges over all the positive roots. The tangent space at $t^{\lambda}$ in $Gr^{\lambda}$ generated by the orbits of the Iwahori subgroup $I$ is 
$(\mathfrak{t} \oplus (\oplus_{\alpha_0} (\mathfrak{g}_{\alpha_0})) \oplus_{k > 0} \mathfrak{g}t^k) / Z$.
The tangent space at $t^{\lambda}$ generated by the orbits of the subgroup $U^-_{\mathcal{O}}$ is 
$(\oplus_{\alpha_0, k \geq 0} \mathfrak{g}_{- \alpha_0} \cdot t^k)/ Z$. As a result, the Iwahori and
$U^-$ orbits generate the tangent space at $t^{\lambda}$ in the $G(\mathcal{O})$ orbit 
$Gr^{\lambda}$.

Therefore the intersection $I^{\lambda} \cap S_{w_0}^{\lambda}$ is transverse. 
By a transversality theorem in \cite{Kle}, it is equi-dimensional and each component has dimension 
$(\height(\lambda_{dom} + \lambda)) + (2 \cdot \height(\lambda_{dom}) - \dim(G/P_{\lambda_{dom}}) + \dim(X_w)) - 2 \cdot \height(\lambda_{dom}) = \height(\lambda_{dom} + \lambda) - \dim(G/P_{\lambda_{dom}}) + \dim(X_w).$

Now we know the theorem holds for the extreme cases $\mu = \lambda_{dom}$ and $\mu = \lambda$. 
Let's consider coweights $\mu$ such that $\lambda < \mu < \lambda_{dom}$. 

In the affine Grassmannian $Gr$, $\overline{S_{w_0}^{\gamma}} = \cup_{\eta \leq \gamma} S_{w_0}^{\eta}$. 
Given a projective embedding of $Gr$, for each $U^-$ orbit $S_{w_0}^{\mu}$, its boundary is given by a hyperplane section $H_{\mu}$. This means that given two coweights $\mu_1$ and $\mu_2$ such that 
$\height(\mu_2) = \height(\mu_1) - 1$, and given any irreducible component $C_2$ of $I^{\lambda} \cap S_{w_0}^{\mu_2}$, there is an irreducible component $C_1$ of $I^{\lambda} \cap S_{w_0}^{\mu_1}$ such that $C_1 \cap H_{\mu_1}$ is dense in $C_2$. Dimension of $C_2$ is bigger than or equal to $\dim(C_1) - 1$, as $C_2$ is cut out by a hyperplane.

The difference of $\dim(S_{w_0}^{\lambda_{dom}} \cap I^{\lambda})$ and $\dim(S_{w_0}^{\lambda} \cap I^{\lambda})$ is exactly $\height(\lambda_{dom} - \lambda)$. 
For $\lambda < \mu \leq \lambda_{dom}$, whenever the height of $\mu$ decreases by $1$, the dimension of the intersection $S_{w_0}^{\mu} \cap I^{\lambda}$ has to also decrease by $1$. 

As a result, we know that for $\lambda \leq \mu \leq \lambda_{dom}$, the intersections $S_{w_0}^{\mu} \cap I^{\lambda}$ are equidimensional and we have the dimension formula as stated in the theorem.

For the emptiness claim, suppose there exists $\mu < \lambda$ such that $I^{\lambda} \cap S_{w_0}^{\mu}$ is nonempty. Then the intersection lies in both $S_{w_0}^{\mu}$ and $S_{w_0}^{\lambda}$. 
However, $S_{w_0}^{\mu} \cap S_{w_0}^{\lambda} = \emptyset$. We have reached a contradiction. 

\end{proof}

By combining the equidimensionality claim in \Cref{IwaMVthrm} and Theorem 11.7 in \cite{Sch}, we prove a result about Demazure modules for the Langlands dual group. 

\begin{corollary}

The number of irreducible components in the intersection of an Iwahori orbit and a $U^-$ orbit in the affine Grassmannian $I^{\lambda} \cap S_{w_0}^{\mu}, \lambda, \mu \in X_*(T)$ is equal to the $\mu-$weight multiplicity in the Demazure module $B^{\lambda}$ for $G^{\vee}$. 

\end{corollary}

\begin{remark}

In the case of MV cycles, the intersection $Gr^{\lambda_{dom}} \cap S^{\mu}$ is nonempty if and only if $\mu$ is in the convex hull of $W \cdot \lambda$, which is also the moment polytope of $\overline{Gr^{\lambda}}$. In our case, let $\lambda = w \cdot \lambda_{dom}$ for some $w \in \tilde{W}$. Then $I^{\lambda} \cap S^{\mu}$ is nonempty if and only if $\mu$ is in the convex hull of $W_{\leq w} \cdot \lambda$. Note that this intersection could be empty even if $\mu$ lies in the moment polytope for $\overline{I^{\lambda}}$, which is bigger than the aforementioned convex hull in general. 

\end{remark}

Now we are going to derive some dimension estimates for generalized MV cycles in the affine flag variety from our understanding of Iwahori MV cycles. In general, there is a large body of literature on emptiness and dimensions of affine Deligne Lusztig varieties, as many basic questions still remain open. For example, see \cite{GHKR, GHKR1, GoHe, He, MST} for some closely related and more general results discovered using different techniques.

\begin{corollary}

Let $\gamma' = (w \cdot \lambda_{dom}, w'), \gamma'' = (\mu, w'') \in W_{aff}$, $\tilde{W}$ denote the quotient of Weyl group associated with the partial flag variety $G/P_{\lambda_{dom}}$, 
$w', w'' \in W, w \in \tilde{W}$, $\lambda_{dom}$ be a dominant coweight and $\mu$ be a coweight. The dimension of $I^{\gamma'} \cap S_{w_0}^{\gamma''}$ in the affine flag variety for $G$ is less than or equal to 

$\height(\lambda_{dom} + \mu) - \dim(G/P_{\lambda_{dom}}) + \dim (X_w) + $

$\left\{ \begin{array}{rcl}
 \dim(G/B) - l(w') & \mbox{if} & w' = w'' \in W \\ 
 l(w') - l(w'') & \mbox{if} & w' > w'' \in W 
\end{array} \right.$

If $w'' > w'$, then the intersection is empty. 

\end{corollary}

\begin{proof}

The intersection $I^{\gamma'} \cap S_{w_0}^{\gamma''}$ lies in the $G/B$ bundle above $I^{w \cdot \lambda_{dom}} \cap S_{w_0}^{\mu} \subset Gr$. 
In the $G/B$ fiber above each point in $I^{w \cdot \lambda_{dom}} \cap S_{w_0}^{\mu} \subset Gr$, a $U^-$ orbit is the same as the 
$B^-$ orbit containing $\gamma''$, and an $I$ orbit is a subset of the product of the $B^-$ orbit and the $B$ orbit containing $\gamma'$. 

When $w' = w''$, the intersection in the $G/B$ fiber above any point is a subset of the $B^-$ orbit containing the T-fixed point indexed by $w' = w''$, whose dimension is given by $\dim(G/B) - l(w')$. 
When $w' > w''$, the intersection in the $G/B$ fiber above any point is a subset of the Richardson variety $X_{w'} \cap X^-_{w''}$, whose dimension is 
$l(w') - l(w'')$. When $w' < w''$, the intersection is empty as the intersections of distinct $B^-$ orbits in $G/B$ are empty. 

\end{proof}

So far we have studied the dimensions of generalized MV cycles in the affine flag variety by focusing on the fact that the affine flag variety is a $G/B$ bundle over the affine Grassmannian. On the other hand, in the future we would like to think more about their moment polytopes in the alcove lattice following the techniques developed in \cite{Kam1}, and extract more geometric information from that. We end this paper by stating a conjecture about Iwahori MV cycles and generalized MV cycles.  

\begin{conj}
Each Iwahori MV cycle and generalized MV cycle is equal to the closure of a GGMS stratum in the affine Grassmannian and affine flag variety respectively. 
\end{conj}

Author's address and email:

Qiao Zhou, Perimeter Institute for Theoretical Physics, 31 Caroline Street North, Waterloo, ON, N2L 2Y5, Canada

qzhou@perimeterinstitute.ca

\end{document}